\documentclass[11pt,reqno,a4paper]{amsart}
\usepackage{amsfonts,amsmath,times}
\usepackage{a4wide}
\usepackage{tikz}

\usefont{T1}{calibri}{b}{n}

\definecolor{blau}{rgb}{0,0,0.75} 
\usepackage[pdftex]{hyperref}
\hypersetup{colorlinks,linkcolor=blau,citecolor=blue,urlcolor=blau}

\allowdisplaybreaks

\newtheorem{theorem}{Theorem}
\newtheorem{lemma}{Lemma}

\newtheorem{prop}{Proposition}
\newtheorem{coroll}{Corollary}
\theoremstyle{definition}
\newtheorem{remark}{Remark}

\newcommand{\fallfak}[2]{\ensuremath{#1^{\underline{#2}}}}
\newcommand{\auffak}[2]{\ensuremath{#1^{\overline{#2}}}}
\newcommand{\Stir}[2]{\genfrac{ \{ }{ \} }{0pt}{}{#1}{#2}}

\newcommand{\C}{\ensuremath{\mathbb{C}}}
\newcommand{\N}{\ensuremath{\mathbb{N}}}
\newcommand{\R}{\ensuremath{\mathbb{R}}}

\newcommand{\Gro}{\ensuremath{\mathcal{O}}}

\newcommand{\law}{\ensuremath{\stackrel{\mathcal{L}}=}}
\newcommand{\claw}{\ensuremath{\xrightarrow{\mathcal{L}}}}

\DeclareMathOperator{\sgn}{\textrm{sgn}}

\DeclareMathOperator{\Arg}{Arg}

\newcommand{\rvR}{\ensuremath{R_{n,\ell}}}
\newcommand{\rvhR}{\ensuremath{\hat{R}_{n,\ell}}}
\newcommand{\rvL}{\ensuremath{L_{n,\ell}}}
\newcommand{\rvY}{\ensuremath{Y_{n,\ell}}}
\newcommand{\rvX}{\ensuremath{X_{n,\ell}}}

\newcommand{\Hank}{\ensuremath{\mathcal{H}}}

\def\P{{\mathbb {P}}}
\def\E{{\mathbb {E}}}

\newcommand{\dzp}{\ensuremath{ \frac{\partial}{\partial z} }}
\newcommand{\dzpZwo}{\ensuremath{ \frac{\partial^2}{\partial z^2} }}

\newcommand{\lo}[1]{\ensuremath{ \log\big( \frac{1}{1-#1}\big)}}
\newcommand{\Lo}[2]{\ensuremath{ \log^{#2}\big( \frac{1}{1-#1}\big)}}

\title{Multiple isolation of nodes in recursive trees}

\author[M.~Kuba]{Markus Kuba}
\address{Markus Kuba\\
Institut f\"ur Angewandte Mathematik und Naturwissenschaften\\
Fachhochschule Technikum Wien\\
H\"ochst\"adtplatz 5, 1200 Wien, Austria}
\email{kuba@dmg.tuwien.ac.at, kuba@technikum-wien.at}

\author[A.~Panholzer]{Alois Panholzer}
\address{Alois Panholzer\\
Institut f{\"u}r Diskrete Mathematik und Geometrie\\
Technische Universit\"at Wien\\
Wiedner Hauptstr. 8-10/104\\
1040 Wien, Austria} \email{Alois.Panholzer@tuwien.ac.at}

\date{\today}

\begin{document}

\begin{abstract}
We introduce the problem of isolating several nodes in random recursive trees
by successively removing random edges, and study the number of random cuts that are necessary for the isolation. 
In particular, we analyze the number of random cuts required to isolate $\ell$ selected nodes in a size-$n$ random recursive tree for three different selection rules, namely $(i)$ isolating all of the nodes labelled $1,2\dots,\ell$ (thus nodes located close to the root of the tree), $(ii)$ isolating all of the nodes labelled $n+1-\ell,n+2-\ell,\dots n$ (thus nodes located at the fringe of the tree), and $(iii)$ isolating $\ell$ nodes in the tree, which are selected at random before starting the edge-removal procedure. Using a generating functions approach we determine for these selection rules the limiting distribution behaviour of the number of cuts to isolate all selected nodes, for $\ell$ fixed and $n \to \infty$.
\end{abstract}

\thanks{The second author was funded by the Austrian Science Foundation FWF, grant P25337-N23.}

\maketitle

\section{Introduction\label{MISOSecIntro}}

Meir and Moon \cite{MeiMoo1970,MeiMoo1974} introduced the following edge-removal procedure for 
cutting down a rooted tree. At each step, pick at random one of the edges; keep the subtree containing the root of the tree, and discard the other subtree.
The main parameter of interest is the number of random cuts necessary to isolate the root.
Meir and Moon studied the random variable $X_{n}$, counting the number of edges that will be removed
from a randomly chosen tree of size $n$ (where the size $|T|$ of a tree $T$ is defined as the number of vertices of $T$) by the above edge-removal
procedure until the root is isolated for two important tree families, namely, for \emph{unordered labelled trees}, also known as \emph{Cayley trees},
and for \emph{recursive trees}, a family of so-called increasingly labelled trees.
For both tree families they obtained exact and asymptotic formul{\ae} for the expectation
$\mathbb{E}(X_{n})$ as well as asymptotic formul{\ae} or bounds, respectively, for the second moment $\mathbb{E}(X_{n}^{2})$.  
Concerning Cayley trees and other families of so-called simply generated trees, a Rayleigh limiting distribution was proven in~\cite{Pan2003,Pan2006} and in a more general setting by Janson~\cite{Jan2006};
Janson also obtained a limit law for complete binary trees~\cite{Jan2004}. Holmgren~\cite{Holmgren2010,Holmgren2011} extended Janson's approach to binary search trees, and more generally to the family of split trees. 
A number of works have analyzed the root isolation process and related processes using the connection of Cayley trees to the so-called Continuum Random Tree,
in particular see the work of Addagio-Berry, Broutin and Holmgren~\cite{ABBH2013} and the recent studies~\cite{Abraham1,Abraham2,Bertoin2012,Bertoin2012+}.

\smallskip

For recursive trees the approach of Meir and Moon was extended in~\cite{Pan2004} and results for all $s$-th moments and $s$-th centered moments of
$X_n$ were obtained. Goldschmidt and Martin~\cite{GoldschmidtMartin2005} related the cutting down procedure to the Bolthausen-Sznitman coalescent. 
Drmota et al.~\cite{Drmota2009} obtained a limiting distribution for $X_n$; the stable limit law was reproven using a probabilistic approach by Iksanov and M\"ohle~\cite{IksanovMoehle2007}.
Moreover, we refer the reader to the work of Bertoin~\cite{Bertoin2012b} for further recent results related to the edge-removal procedure.

\subsection{Node isolation in labelled trees}
There exist some works that generalize the edge-removal procedure of Meir and Moon for rooted trees to isolate non-root nodes.
In \cite{PanKu2005II} the reverse procedure, where the subtree containing the root is discarded, was studied for several important tree families.
Furthermore, in \cite{PanKu2005} the random variable $\rvX$ was studied, where $\rvX$ counts the number of random cuts
necessary to isolate the node labelled $\ell$, with $1 \le \ell \le n$, in a random size-$n$ recursive tree. 

In the present work we want to examine the behaviour of the edge-removal procedure when using it to isolate simultaneously a number of specified nodes in the tree.
Thus, in the following we consider a general edge-removal procedure for \emph{labelled} trees,
where we always assume that the labels $1,2,\dots,n$ are distributed amongst the $n$ nodes of a tree of size $n$ (furthermore, we will always identify a node with its label).
Namely, given a tree $T$ of size $n$ and a set of labels $\lambda_1,\dots,\lambda_\ell$, with $1 \le \lambda_{1} < \lambda_{2} < \dots < \lambda_{\ell-1} < \lambda_{\ell} \le n$ and $1\leq \ell \le n$, we will isolate the nodes $\lambda_1,\dots,\lambda_\ell$ as follows. 
We start by picking one of the $n-1$ edges of the tree $T$ uniformly at random (i.e., each edge in the tree might be chosen equally likely and independently of the labels of the nodes we are going to isolate) and removing it. This separates the tree $T$ into a pair of rooted trees; the tree containing the root of the original tree, let us call it $B'$, retains
its root, while the other tree, let us denote it by $B''$, is rooted at the vertex adjacent to the edge that was cut.
If one of these trees $B', B''$ does not contain any of the nodes $\lambda_1,\dots,\lambda_\ell$, we discard it and only keep the other one, otherwise we keep both of them.
Then we continue this procedure to the one or two remaining trees. In general, when we have a forest $F$ consisting of $m$ rooted trees $B_{1}, \dots, B_{m}$ we pick at random one of the edges of $F$ and remove it.
Let us assume this edge is contained in the tree $B_{j}$. Then $B_{j}$ is separated into a pair of rooted trees $B_{j}'$ (containing the root of $B_{j}$) and $B_{j}''$. Again, 
if either $B_{j}'$ or $B_{j}''$ does not contain any of the nodes $\lambda_1,\dots,\lambda_\ell$, we discard it and only keep the other one, otherwise we keep both of them, which, together with the remaining trees $B_{1}, \dots, B_{j-1}, B_{j+1}, \dots, B_{m}$, form a new forest. We continue this procedure until all nodes $\lambda_1,\dots,\lambda_\ell$ are isolated, i.e., until we get a forest consisting of $\ell$ trees, which are the $\ell$ isolated vertices $\lambda_{1}, \dots, \lambda_{\ell}$.
This generalized edge-removal procedure is illustrated in Figure~\ref{fig1}.
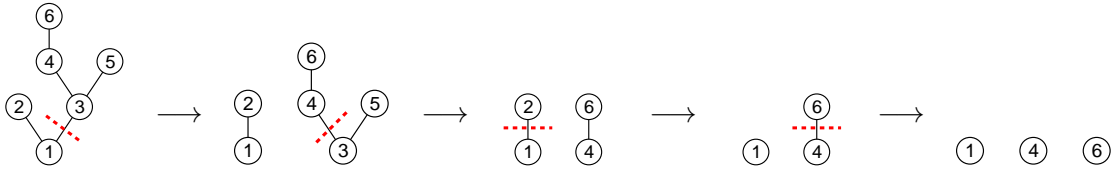
\begin{figure}
\begin{center}
\setlength{\unitlength}{1cm}
\begin{picture}(15,3)
\put(0,0){\includegraphics[width=2cm]{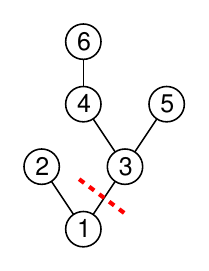}}
\put(3,0){\includegraphics[width=2.5cm]{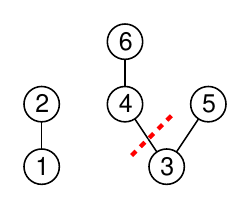}}
\put(6.5,0){\includegraphics[width=2cm]{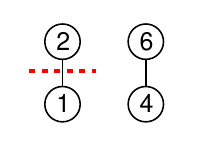}}
\put(9.5,0){\includegraphics[width=2cm]{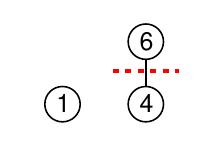}}
\put(12.5,0){\includegraphics[width=2.5cm]{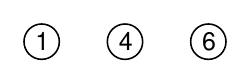}}
\put(2.2,0.8){$\longrightarrow$}
\put(5.7,0.8){$\longrightarrow$}
\put(8.7,0.8){$\longrightarrow$}
\put(11.7,0.8){$\longrightarrow$}
\end{picture}
\caption{Isolating the nodes $1,4,6$ in a recursive tree of size $6$ via the edge-removal procedure by using $4$ cuts.\label{fig1}}
\end{center}
\end{figure}

We are going to study this edge-removal procedure for random recursive trees. A labelled rooted unordered tree $T$ (i.e., there is no left-to-right ordering on the subtrees of any node) of size $n$ is called a recursive tree, if the labels amongst the path from the root node to any node $v \in T$ are always forming an increasing sequence (thus, the family of recursive trees consists of all increasingly labelled unordered trees). It is well-known and easy to show that there are exactly $T_{n} := (n-1)!$ different size-$n$ recursive trees. When we pick one of these $(n-1)!$ recursive trees at random we speak about a random recursive tree of size $n$. Random recursive trees can be generated by a simple growth rule: a random tree $T$ of size $n$ is obtained from a random tree $\tilde{T}$ of size $n-1$ by choosing uniformly at random a node in $\tilde{T}$ and attaching the node labelled $n$ to it.

In contrast to Cayley trees, where the labels are distributed uniformly amongst the nodes of a tree, the label of a node has a strong influence on its expected location in a recursive tree; e.g.,
the depth, i.e., the root-to-node distance, of node $j$ is (for $j \to \infty$) normally distributed with expectation and variance $\sim \log j$ (see, e.g., \cite{MahSmy1995,PanPro2007}).
Thus, we are particularly interested in the influence of the labels of the selected nodes to the general edge-removal procedure and study its behaviour when isolating in a random recursive tree of size $n$ the first $\ell$ inserted nodes, i.e., the nodes labelled $1,2,\dots,\ell$, which are all located near the root, and when isolating the last $\ell$ inserted nodes labelled $n+1-\ell,\dots,n$, which are all located at the fringe. We denote with $X_{n; (\lambda_{1},\lambda_{2},\dots, \lambda_{\ell})}$ the random variable counting the total number of random cuts necessary to isolate the nodes $\lambda_{1},\lambda_{2},\dots, \lambda_{\ell}$, with $1 \le \lambda_{1} < \lambda_{2} < \dots < \lambda_{\ell} \le n$ and $1 \le \ell \le n$ in a random size-$n$ recursive tree and introduce the random variables
\begin{equation*}
R_{n,\ell} := X_{n;(1,\dots,\ell)} \quad \text{and} \quad L_{n,\ell} := X_{n;(n+1-\ell,\dots,n)}.
\end{equation*}
$R_{n,\ell}$ is thus counting the number of removed edges until the nodes labelled
$1,2,\dots,\ell$ (nodes close to the root, for $\ell$ fixed) are isolated and $L_{n,\ell}$ is counting the number of removed edges until the nodes labelled
$n+1-\ell,\dots,n$ (nodes at the fringe and which are leaves with high probability, for $\ell$ fixed) are isolated.

Furthermore, we are interested in the behaviour of the general edge-removal procedure when isolating $\ell$ randomly selected nodes in a random recursive tree of size $n$, i.e., where $\ell$ labels $\lambda_{1}, \lambda_{2}, \dots, \lambda_{\ell}$ are selected uniformly at random amongst all $\binom{n}{\ell}$ subsets of size $\ell$ of $\{1, \dots, n\}$ and the edge-removal procedure isolates these selected nodes in a random recursive tree of size $n$. Let us denote by $U(n,\ell)$ a r.v.\ uniformly distributed on the subsets of $\{1, \dots, n\}$ of size $\ell$, i.e.,
$\mathbb{P}\{U(n,\ell) = (\lambda_{1}, \dots, \lambda_{\ell})\} = \frac{1}{\binom{n}{\ell}}$, for $1 \le \lambda_{1} < \lambda_{2} < \dots < \lambda_{\ell} \le n$ and $1 \le \ell \le n$.
We introduce the random variable
\begin{equation*}
Y_{n,\ell} := X_{n;U(n,\ell)},
\end{equation*}
which counts the number of removed edges until $\ell$ randomly selected nodes are isolated in a random recursive tree of size $n$.

In this work we analyze the limiting distribution behaviour of the random variables $\rvR$, $\rvL$ and $\rvY$, for a fixed number $\ell$ of selected nodes and the tree-size $n$ tending to infinity, 
by treating the distributional recurrences of $\rvR$, $\rvL$ and $\rvY$, respectively, by means of a generating functions approach and applying complex-analytic techniques. For all of these quantities we are able to provide limit laws and state asymptotic expansions of the integer moments. 


\subsection{Notation}
Throughout this paper we use the abbreviations $\fallfak{x}{k}:= x (x-1) \cdots (x-k+1)$ and $\auffak{x}{k} := x (x+1) \cdots
(x+k-1)$ for the falling and rising factorials, respectively. We use the notation $\Stir{s}{j}$ for the Stirling numbers of the second kind, 
appearing in the formula to convert falling factorials into powers: $x^{s}=\sum_{j=0}^{s}\Stir{s}{j} \fallfak{x}{j}$. Furthermore, we use the abbreviations $D_{x}$ for the differential operator with respect to $x$ and $E_{x}$ for the evaluation operator at $x=1$. Moreover, we denote by $X \law Y$ the equality in distribution of the random variables $X$ and $Y$, and by $X_{n} \claw X$ convergence in distribution of the sequence of random variables $X_{n}$ to a random variable $X$.

\subsection{Auxiliary results about probability distributions}
For the readers convenience we collect a few basic facts about two probability distributions appearing later in our analysis.

A beta-distributed random variable $Z \law \beta(\alpha,\beta)$ with parameters $\alpha,\beta>0$ has 
a probability density function given by $f(x)=\frac1{B(\alpha,\beta)}x^{\alpha-1}(1-x)^{\beta-1}$, where 
$B(\alpha,\beta)=\frac{\Gamma(\alpha)\Gamma(\beta)}{\Gamma(\alpha+\beta)}$ denotes the Beta-function.
The (power) moments of $Z$ are given by 
\[
\E(Z^s)=\frac{\prod_{j=0}^{s-1}(\alpha+j)}{\prod_{j=0}^{s-1}(\alpha+\beta+j)}
=\frac{\fallfak{(\alpha+s-1)}{s}}{\fallfak{(\alpha+\beta-1)}{s}},\quad s\ge 1.
\]
The beta-distribution is uniquely determined by the sequence of its moments. In this work we will discuss
a beta-distributed random variable $Z \law \beta(\ell,1)$, with moments given by
$\E(Z^s)=\frac{\ell}{\ell+s}$, for $s\ge 1$.

\smallskip

A random variable $R$ with cumulative distribution function $\nu$ is called stable, if its characteristic function is given by 
\[
\E(e^{it R})=\int_\R e^{itx}\nu(dx)= \exp\Big(it\mu - c |t|^\alpha\,(1\,-\,i \beta\,\sgn (t) \omega(t,\alpha))\Big),\quad t\in\R,
\]
where $\mu\in\R$, 
$c>0$, 
$\beta\in [-1, 1]$, 
the so-called exponent of stability $\alpha\in (0,2]$, and
\[
 \omega(t,\alpha)=
 \begin{cases}
 \tan(\frac{\pi \alpha}2),\quad\text{if } \alpha\neq 1,\\
 -\frac{2}{\pi}\log|t|, \quad\text{if } \alpha=1.
 \end{cases}
\]
A stable distribution $\nu$ is uniquely determined by the generating quadruple $(\mu, c, \alpha,\beta)$, and 
it is known that $\nu$ is either degenerate, normal, or has the so-called L\'evy spectral function
$M(x)$ of the form 
\[
M(x)=c_1|x|^{-\alpha},\quad x<0,\qquad
M(x)=-c_2|x|^{-\alpha},\quad x>0,
\]
where $c_1,c_2\ge 0$ and $c_1+c_2>0$. In this work we will have a stable distribution 
with index of stability $\alpha=1$, and generating quadruple $(\mu, c, \alpha,\beta)=(0,\frac{\pi}2,1,-1)$, 
such that the characteristic function satisfies
\[
\E(e^{it R})=e^{-\frac{\pi}2(|t|(1-\frac{2i}{\pi}\log|t|\sgn(t)))}
=e^{i t\log|t|-\frac{\pi}2|t|}.
\]
Since the constants $c_1,c_2$ are related to $\alpha,\beta$ in the case $\alpha=1$ by the equations
$\beta=\frac{c_2-c_1}{c_2+c_2}$, and $c=\frac\pi2(c_1+c_2)$, we
observe that $c_1=1$ and $c_2=0$; thus, the distribution of $R$ is spectrally negative. Note that $R$
arises as a limiting distribution for the discrete Luria-Delbr\"uck distribution, 
and is sometimes called the continuous Luria-Delbr\"uck distribution~\cite{Moehle}.

\subsection{Plan of the paper}
In the next section we present our results concerning the limiting behaviour of the random variables $\rvR$, $\rvL$ and $\rvY$. In Section~\ref{MISOSecPrelim} we use basic combinatorial considerations
to derive the splitting probabilities for the sizes of the trees occurring after a random cut, and set up distributional equations for the random variables of interest. Section~\ref{MISOSecEinsEll} is concerned with the analysis of $\rvR$: we determine a closed form expression for a suitably defined generating function, which allows to deduce a limit law by using complex-analytic techniques.
Section~\ref{MISOSecEnnPlusEinsMinusEll} is devoted to an analysis of $\rvL$, where we use again generating functions, but now in combination with an inductive approach to extract the asymptotic behaviour of all integer moments.
Finally, Section~\ref{MISOSecRandom} shows the results for $\rvY$ with a similar approach.

\section{Results\label{MISOSecResults}}

\begin{theorem}
   \label{MISOthe1stable}
The normalized random variable 
\begin{equation*}
R^*_{n,\ell} := \frac{\rvR-(\ell-1)-\frac{n}{\log n} - \frac{n\log\log n}{(\log n)^2}}{\frac{n}{(\log n)^2}}
\end{equation*}
of $\rvR$ counting the number of random cuts necessary to isolate nodes $1,\dots,\ell$ in a random recursive tree of size $n$
converges, for arbitrary but fixed $\ell\in\N$ and $n \to \infty$, weakly to a stable random variable $R^*$ with characteristic function
\begin{equation*}
\varphi_{R^*}(t) := \mathbb{E}\big(e^{i t R^*}\big) = e^{i t \log|t|-\frac{\pi}{2}|t|}.
\end{equation*}
\end{theorem}

\noindent{\textbf{Remark:}} Using the explicit form of the generating function $M_{\ell}(z,v)$ introduced and studied in Section~\ref{MISOSecEinsEll}
we can further show that the $s$-th integer moment $\E(\rvR^s)$, $s\geq 1$, of $\rvR$ is, for $\ell\in\N$ fixed and $n \to \infty$, asymptotically given by
\begin{equation*}
   \E(\rvR^s) = \frac{n^s}{\log^s n} + \mathcal{O}\Big(\frac{n^s}{\log^{s+1}{n}}\Big).
\end{equation*}
This implies that a non-degenerate limiting distribution result cannot be obtained from the moment's sequence, since $\frac{\log n}{n} \rvR \claw 1$. 
Thus, we omit the computations concerning the $s$-th moments of $\rvR$.



\bigskip

\begin{theorem}
   \label{MISOthe2}
   The $s$-th integer moment $\E(\rvL^{s})$, $s\geq 1$, of the number of random cuts necessary to
   isolate the nodes $n+1-\ell,\dots,n$ in a random recursive tree of size $n$ is, for arbitrary but fixed $\ell\in\N$ and $n \to \infty$, asymptotically given by
   \begin{equation*}
      \E(\rvL^s) = \frac{\ell}{s+\ell} \cdot \frac{n^s}{\log^s n} + \mathcal{O}\Big(\frac{n^s}{\log^{s+1}{n}}\Big).
   \end{equation*}
   Thus, the normalized random variable $\frac{\log n}{n}\rvL$ converges in distribution to a beta-distributed random variable $\beta(\ell,1)$ with parameters $\ell$ and $1$,
   \begin{equation*}
   \frac{\log n}{n}\rvL \claw \beta(\ell,1).
   \end{equation*}
\end{theorem}


\bigskip

\begin{theorem}
   \label{MISOthe3}
   The $s$-th integer moment $\E(\rvY^{s})$, $s\geq 1$, of the number of random cuts necessary to
   isolate $\ell$ randomly selected nodes in a random recursive tree of size $n$ is, for arbitrary but fixed $\ell\in\N$ and $n \to \infty$, asymptotically given by
   \begin{equation*}
      \E(\rvY^s) = \frac{\ell}{s+\ell} \cdot \frac{n^s}{\log^s n} + \mathcal{O}\Big(\frac{n^s}{\log^{s+1}{n}}\Big).
   \end{equation*}
   Thus, the normalized random variable $\frac{\log n}{n}\rvY$ converges in distribution to a beta-distributed random variable $\beta(\ell,1)$ with parameters $\ell$ and $1$,
   \begin{equation*}
   \frac{\log n}{n}\rvY \claw \beta(\ell,1).
   \end{equation*}
\end{theorem}

\section{Preliminaries\label{MISOSecPrelim}}

First let us consider the procedure for isolating nodes $1, 2, \dots, \ell$ of a given random recursive tree $T$ of size $|T|=n$ via random cuts.
After removing a randomly chosen edge of $T$, it splits into two subtrees $T^{(1)}$ and $T^{(2)}$, where we assume that $T^{(1)}$ contains the node labelled $\ell$.
A very important property of recursive trees that allows the approach presented is the \emph{randomness preservation property\footnote{This property is called splitting property by Bertoin~\cite{Bertoin2012b}; see Panholzer \cite{Pan2006} for a characterization of all simply generated trees possessing this property and~\cite{Bertoin2012b} for a recent discussion of this attribute.}}: both subtrees $T^{(1)}, T^{(2)}$ are, after an order-preserving relabelling with labels $1, 2, \dots, |T^{(1)}|$ and $1, 2, \dots, |T^{(2)}|$, respectively, again random recursive trees of respective sizes. In order to setup a distributional equation for the random variable $\rvR$ we have to keep track of the sizes of $T^{(1)}$ and $T^{(2)}$, respectively, after the edge-removal.
Moreover, we have to take into account the distribution of the nodes labelled $1,2,\dots,\ell$ in the original tree $T$ over the two subtrees $T^{(1)}$ and $T^{(2)}$.
To do this we use purely combinatorial arguments to derive in Subsection~\ref{MISOSSecSplitProb} the so-called \emph{splitting probabilities} $p_{(n,\ell),(k,r)}$, which give the probability that, when starting with a random size-$n$ recursive tree and removing a random edge, the subtree containing node $\ell$ is of size $k$ and where furthermore node $\ell$ is the $r$-th smallest node in this subtree.
Formul\ae\, for $p_{(n,\ell),(k,r)}$ already occurred in \cite{PanKu2005}, but in order to keep the present work self-contained we reproduce a slightly adapted proof of them.
These splitting probabilities readily yield a distributional equation for $\rvR$ as stated in Subsection~\ref{MISOSSecDistEqn}.

For the problem of isolating nodes $n+1-\ell, \dots, n$ in $T$ the situation is very similar, but one has to keep track of node $n+1-\ell$ (instead of node $\ell$) in the occurring subtrees $T^{(1)}$ and $T^{(2)}$ after a random cut and to take into account the distribution of the nodes labelled $n+1-\ell, \dots, n$ in the original tree $T$ over the subtrees.
Again, by using the splitting probabilities $p_{(n,\ell),(k,r)}$ a distributional equation for $\rvL$ can be established, which is carried out in Subsection~\ref{MISOSSecDistEqn}.

When isolating $\ell$ randomly selected nodes in $T$ the situation is considerably easier and in order to state a distributional equation for $\rvY$ it suffices to know the splitting probabilities $p_{n,k}$, which give the probability that, when removing a random edge of a random size-$n$ recursive tree, the subtree containing the original root node is of size $k$ (whereas the other one is of size $n-k$). These probabilities $p_{n,k}$ have been computed already in \cite{MeiMoo1974}; however, they also occur as a special instance of the more general splitting probabilities $p_{(n,\ell),(k,r)}$, since it holds $p_{n,k} = p_{(n,1),(k,1)}$ due to the fact that the root node (label $1$) of the original tree is in any case the smallest node in the corresponding subtree.
For the sake of completeness we state in Subsection~\ref{MISOSSecSplitProb} the probabilities $p_{n,k}$, which are then used in Subsection~\ref{MISOSSecDistEqn} to deduce a distributional equation for $\rvY$.

\subsection{Splitting probabilities\label{MISOSSecSplitProb}}
\begin{lemma}[~\cite{PanKu2005}]
   \label{MISOlem1}
   The splitting probabilities $p_{(n,\ell),(k,r)}$ are, for $1 \le \ell \le n$,
   $1 \le r \le k$, $1 \le k \le n-1$ and $n \ge 2$, given as follows:
   \begin{align*}
      & p_{(n,\ell),(k,r)} =
      \begin{cases}
         & \left[(\ell-1)\binom{n-\ell}{n-k} + \binom{n-\ell+1}{n-k+1}\right]
         \frac{(k-1)! (n-k-1)!}{(n-1)(n-1)!}, \quad r=\ell, \\[3mm]
         & \left[ \binom{\ell-1}{r} \binom{n-\ell}{k-r} + \binom{\ell-1}{r-2} \binom{n-\ell}{k-r}\right]
         \frac{(k-1)! (n-k-1)!}{(n-1)(n-1)!}, \quad r<\ell.
      \end{cases}
   \end{align*}
\end{lemma}
\begin{proof}
If we remove an edge $e$ of a size-$n$ recursive tree $T$ we split the tree into two subtrees:
we denote with $B'$ the subtree containing the original root, i.e., label $1$,
and with $B''$ the other subtree, which is rooted at the vertex adjacent to the edge $e$ that was cut.
After an order-preserving relabelling with labels $\{1, \dots, |B'|\}$ and $\{1, \dots, |B''|\}$, respectively, both subtrees can be considered as recursive trees. Furthermore we denote with $T^{(1)}$ the arising subtree, which contains the node labelled by $\ell$ in the original tree,
and with $T^{(2)}$ the other subtree; we assume that this subtree $T^{(1)}$ has size $k$, with $1 \le k \le n-1$, and that it contains exactly $r$ nodes of the set
$\{1,2,\dots,\ell\}$ including the node labelled $\ell$. Apparently this implies that the tree $T^{(2)}$ is of size $n-k$ and contains $\ell-r$ nodes of the set $\{1,2,\dots,\ell\}$.
We distinguish now the cases $r=\ell$ and $1 \le r < \ell$.

If $r=\ell$ then it follows that $T^{(1)}=B'$, since all nodes labelled $1,2,\dots,\ell$ have to be contained in $T^{(1)}$, 
in particular the original root labelled $1$. We want to determine the number of possibilities of removing an edge $e$ of a recursive tree of size $n$
leading (after an order-preserving relabelling) to the pair $(B', B'')$ of subtrees.
To do this we count the number of different ways of distributing the labels $\{1, \dots, n\}$ order-preserving
over $B'$ and $B''$ and adjoining the root of $B''$ to a node of $B'$ (by inserting edge $e$),
such that the resulting tree is a recursive tree.
We consider now the node of $B'$ incident with $e$: if the node of $B'$ incident with $e$ has label $j$,
with $1 \le j \le k$, then it follows that
the labels of $B''$ must all be larger than $j$.
For $1\leq j \leq \ell$ we can choose $n-k$ of the labels $\ell+1, \ell+2, \dots, n$
and distribute them order-preserving over $B''$,
whereas the remaining labels are distributed order-preserving over $B'$,
leading to $\binom{n-\ell}{n-k}$ possibilities.
For $\ell+1 \leq j \leq k$ we can choose $n-k$ of the labels $j+1, j+2, \dots, n$
and distribute them order-preserving over $B''$,
whereas the remaining labels are distributed order-preserving over $B'$,
leading to $\binom{n-j}{n-k}$ possibilities.
Thus this quantity is independent of the actual choice of $B'$ with $|B'| = k$ and $B''$ with $|B''| = n-k$.
Since there are $T_{k} = (k-1)!$ and $T_{n-k} = (n-k-1)!$ different recursive trees of size $k$ and $n-k$, respectively,
this leads, together with the fact that there are $n-1$ ways of selecting an edge $e$ for any of the $T_{n} = (n-1)!$
recursive trees of size $n$, to the following formula:
\begin{align*}
p_{(n,\ell),(k,\ell)} & =  \left[\ell\binom{n-\ell}{n-k} + \sum_{j=\ell+1}^{k}\binom{n-j}{n-k}%
                     \right]\frac{(k-1)!(n-k-1)!}{(n-1)(n-1)!}\\
                & =  \left[(\ell-1)\binom{n-\ell}{n-k} + \binom{n-\ell+1}{n-k+1}\right]%
                     \frac{(k-1)! (n-k-1)!}{(n-1)(n-1)!},
\end{align*}
appealing to a well known identity.

If $r < \ell$ we have to distinguish further between the two cases $T^{(1)}=B'$ and $T^{(1)}=B''$. If $T^{(1)}=B'$
and we distribute the labels $\{1, \dots, n\}$ order-preserving over $B'$ and $B''$,
we have the restriction that exactly $\ell-r$ nodes of the nodes $2,\dots,\ell-1$ have to be in $B''$.
If $T^{(1)}=B''$ then we have the restriction that exactly $r-1$ nodes of the nodes $2,\dots,\ell-1$ have to be
in $B''$. Proceeding the same way as before we obtain eventually the following formula.
\begin{align*}
p_{(n,\ell),(k,r)} & =  \left[\binom{n-\ell}{n-k - (l-r)}\sum_{j=1}^{r-1}\binom{\ell-1-j}{\ell-r}%
                     + \binom{n-\ell}{k-r}\sum_{j=1}^{\ell-r}\binom{\ell-1-j}{r-1}\right]\\
                     & \quad \times \frac{(k-1)!(n-k-1)!}{(n-1)(n-1)!}\\
                & =  \left[\binom{\ell-1}{r-2}\binom{n-\ell}{k-r} %
                     + \binom{\ell-1}{r}\binom{n-\ell}{k-r} \right]%
                     \frac{(k-1)! (n-k-1)!}{(n-1)(n-1)!}
\end{align*}%
\end{proof}
The particular instance $\ell=1$ and $r=1$ in Lemma~\ref{MISOlem1} rederives the well-known formula for the splitting probabilities $p_{n,k}$.
\begin{coroll}[~\cite{MeiMoo1974}]
   \label{MISOcor1}
   The splitting probabilities $p_{n,k}$ are, for $1 \le k \le n-1$ and $n \ge 2$, given as follows:
   \begin{equation*}
      p_{n,k} = \frac{n}{(n-1) (n-k+1) (n-k)}.
   \end{equation*}
\end{coroll}

\subsection{Distributional equations\label{MISOSSecDistEqn}}
Using the splitting probabilities $p_{(n,\ell),(k,r)}$ given in Subsection~\ref{MISOSSecSplitProb} we can readily set up a distributional equation for 
the random variable $\rvR$. In this context it is appropriate to define also the r.v.\ $R_{n,0}$ (i.e., the number of cuts to isolate $0$ nodes) via $R_{n,0} = 0$, for $n \ge 1$.
When considering a random recursive tree of size $n \ge 2$ and eliminating a random edge one immediately gets
\begin{equation}
\label{MISOdistributionalR1}
\rvR \law R^{(1)}_{I_n,J_\ell} + R^{(2)}_{n-I_n,\ell-J_\ell}+1,\quad \text{for $n\ge 2$ and $1\le \ell\le n$},
\end{equation}
where $I_n$ counts the size of the subtree containing the node originally labelled $\ell$ after removing a random edge, and $J_\ell$ counts the number of nodes of the set $\{1,2,\dots,\ell\}$ contained in
this subtree. The random variables $(R_{k,r}^{(j)})_{k\ge 1,1\le r\le k}$, $j=1,2$, have the same distribution as $(R_{n,\ell})_{n\ge 1, 1\le \ell\le n,}$, and the variables $I_n$, $J_\ell$
are independent of $(R_{k,r}^{(j)})_{k\ge 1,1\le r\le k}$, $j=1,2$. The initial value is given by $R_{1,1}=0$. 
Further note that by combinatorial reasoning it is apparent that $R_{\ell,\ell}=\ell-1$, for $\ell \ge 1$, since a tree of size $\ell$ contains exactly $\ell-1$ edges, which all have to be eliminated.

To benefit from recurrence \eqref{MISOdistributionalR1} one requires the joint distribution of $I_n$ and $J_\ell$, for $1 \le \ell \le n$,
$1 \le r \le k$, $1 \le k \le n-1$ and $n \ge 2$, but due to previous considerations this is exactly given by the splitting probabilities, i.e.,
\begin{equation*}
\P\{I_n=k,J_\ell=r\}= p_{(n,\ell),(k,r)}.
\end{equation*}
After some simplifications one gets the following expression, which is advantageous for further computations:
\begin{equation}
\label{MISOdistributionalR2}
\P\{I_n=k,J_\ell=r\}= p_{(n,\ell),(k,r)}=
    \begin{cases}
         & \left[\ell-1 + \frac{n-\ell+1}{n-k+1} \right]
         \frac{\fallfak{(k-1)}{\ell-1}}{(n-1)(n-k)\fallfak{(n-1)}{\ell-1}}, \quad r=\ell, \\[3mm]
         & \left[ \binom{\ell-1}{r}  +  \binom{\ell-1}{r-2} \right]
         \frac{ \fallfak{(k-1)}{r-1}\fallfak{(n-k-1)}{\ell-r-1}}{(n-1)\fallfak{(n-1)}{\ell-1}}, \quad r<\ell.
      \end{cases}
\end{equation}

\medskip

Analogeously, the random variable $\rvL$ satisfies a distributional equation similar to $\rvR$, where again it is appropriate to introduce also the r.v.\ $L_{n,0}$  via $L_{n,0} = 0$:
\begin{equation}
\label{MISOdistributionalL1}
\rvL \law L^{(1)}_{\hat{I}_n,\hat{J}_\ell} + L^{(2)}_{n-\hat{I}_n,\ell-\hat{J}_\ell}+1, \quad \text{for $n\ge 2$ and $1\le \ell\le n$},
\end{equation}
where $\hat{I}_n$ counts the size of the subtree containing the node originally labelled $n+1-\ell$ after removing a random edge, and $\hat{J}_\ell$ counts the number of nodes of the set $\{n+1-\ell,n+2-\ell,\dots,n\}$ contained in this subtree. Here, again the random variables on the right hand side are independent copies of $\rvL$ that are independent of the variables $\hat{I}_n$, $\hat{J}_\ell$. The initial value is given by $L_{1,1}=0$. 

The joint distribution of the random variables $\hat{I}_n$ and $\hat{J}_\ell$ is, for $1 \le \ell \le n$,
$1 \le r \le k$, $1 \le k \le n-1$ and $n \ge 2$, again determined by the splitting probabilities via
\begin{equation*}
\P\{\hat{I}_n=k,\hat{J}_\ell=r\}= p_{(n,n+1-\ell),(k,k+1-r)}.
\end{equation*}
In succeeding computations we will use the following expression, which is obtained after some simplifications:
\begin{equation}
\label{MISOdistributionalL2}
  \begin{split}
\P\{\hat{I}_n=k,\hat{J}_\ell=r\}&= 
           \frac{\binom{\ell-1}{r-1}\big[ \fallfak{(k-1)}{r-2}\fallfak{(n-k-1)}{\ell-r} %
        + \fallfak{(k-1)}{r}\fallfak{(n-k-1)}{\ell-r-2}\big]}{(n-1)\fallfak{(n-1)}{\ell-1}}%
        \\
    & +  \delta_{k,n+r-\ell}\binom{\ell}{r-1}\frac{(\ell-r-1)!}{(n-1)\fallfak{(n-1)}{\ell-r}}
         , \quad r\leq \ell,
      \end{split}
\end{equation}
where $\fallfak{(j-1)}{-p} :=(\auffak{j}{p})^{-1}$, for $p\in\N$, and $\delta$ denotes the Kronecker-delta function.

\medskip

Finally, the random variable $\rvY$ satisfies the following distributional equation, with $Y_{n,0} = 0$, for $n \ge 1$, and the initial value $Y_{1,1}=0$:
\begin{equation}
\label{MISOdistributionalY1}
\rvY \law Y^{(1)}_{\tilde{I}_n,\tilde{J}_\ell} + Y^{(2)}_{n-\tilde{I}_n,\ell-\tilde{J}_\ell}+1, \quad \text{for $n\ge 2$ and $1\le \ell\le n$},
\end{equation}
where $\tilde{I}_n$ counts the size of the subtree containing the original root of the tree after removing a random edge, and $\tilde{J}_\ell$ counts the number of selected nodes, which shall be isolated, contained in this subtree. The random variables on the right hand side are independent copies of $\rvY$ that are independent of the variables $\tilde{I}_n$, $\tilde{J}_\ell$. 

The joint distribution of the random variables $\tilde{I}_n$ and $\tilde{J}_\ell$ is then, for $1 \le \ell \le n$,
$1 \le r \le k$, $1 \le k \le n-1$ and $n \ge 2$, determined by the splitting probabilities $p_{n,k}$ via
\begin{equation}
\label{MISOdistributionalY2}
\P\{\tilde{I}_n=k,\tilde{J}_\ell=r\}= \frac{\binom{k}{r} \binom{n-k}{\ell-r}}{\binom{n}{\ell}} \cdot p_{n,k}
= \frac{\binom{k}{r} \binom{n-k}{\ell-r}}{\binom{n}{\ell}} \cdot \frac{n}{(n-1) (n-k+1)(n-k)}.
\end{equation}

\section{Isolating the nodes \texorpdfstring{$1,\dots,\ell$}{1,...,l}\label{MISOSecEinsEll}}

\subsection{Deriving suitable generating functions solutions}
To treat the distributional equation \eqref{MISOdistributionalR1} for the r.v.\ $\rvR$ we will, for $\ell \ge 1$, introduce suitable generating functions via
\begin{equation}
\label{MISOgen1}
   M_{\ell}(z,v) := \sum_{n \ge \ell}(n-1)^{\underline{\ell-1}} \, \E(v^{\rvR})z^{n-\ell} = \sum_{n \ge \ell}\sum_{m \ge 0} (n-1)^{\underline{\ell-1}} \, \mathbb{P}\{\rvR=m\} z^{n-\ell} v^{m}. 
\end{equation}

This yields a description of the problem by means of a differential equation, which turns out to be very useful later on.

\begin{prop}\label{MISOpropMellzv}
The generating function $M_\ell(z,v)$ satisfies for $\ell \ge 1$ the following first order linear differential equation:
\begin{equation}
  \label{MISOode1}
  \dzp M_{\ell}(z,v) - (\ell f(z,v)-(\ell-1)g(z,v))M_{\ell}(z,v) = g(z,v)\cdot b_\ell(z,v),
\end{equation}
with functions
\begin{equation}
\label{MISOweak1}
f(z,v)=\frac{v\lo{z}}{(1-z)v \lo{z} + z(1-v)}, \quad g(z,v)=\frac{1}{(1-z)v \lo{z} + z(1-v)},
\end{equation}
and 
\begin{equation}
\label{MISObellzv}
   b_\ell(z,v) = v \sum_{r=1}^{\ell-1} \left[ \binom{\ell-1}{r} + \binom{\ell-1}{r-2}\right]
   M_{r}(z,v) M_{\ell-r}(z,v),
\end{equation}
and the initial condition $M_{\ell}(0,v) = v^{\ell-1}(\ell-1)!$. 
\end{prop}
\begin{proof}
{}From the distributional equation \eqref{MISOdistributionalR1} we immediately
obtain the following recurrence for the probability generating function of $\rvR$:
\begin{align}
\label{MISOrec1}
   \E(v^{\rvR}) = v\sum_{r=1}^{\ell} \sum_{k=r}^{n-1}\P\{I_n=k,J_\ell=r\} \E(v^{R^{(1)}_{k, r }})\E(v^{R^{(2)}_{n-k, \ell-r }}), \quad \text{for $1 \le \ell \le n$ and $n \ge 2$},
\end{align}
with initial value $\E(v^{R_{1,1}})= 1$ and where the probabilities $\P\{I_n=k,J_\ell=r\} = p_{(n,\ell),(k,r)}$ are given in~\eqref{MISOdistributionalR2}.
Multiplying recurrence~\eqref{MISOrec1} by $(n-1)\fallfak{(n-1)}{\ell-1}z^{n-\ell}$ and taking the summation over all $n\geq \ell$ leads to the differential equation (we omit here these lengthy, but straightforward computations)
\begin{equation*}
  \Big((1-z)v \lo{z} + z(1-v)\Big)\dzp M_{\ell}(z,v)  +  \Big(\ell-1 -\ell v \lo{z} \Big) M_{\ell}(z,v) = b_\ell(z,v),
\end{equation*}
with $b_{\ell}(z,v)$ given by \eqref{MISObellzv}. Simple manipulations and using \eqref{MISOweak1} yield the stated differential equation~\eqref{MISOode1}. 
Note that $M_{\ell}(0,v) =(\ell-1)^{\underline{\ell-1}} \E(v^{R_{\ell,\ell}})= v^{\ell-1}(\ell-1)!$, since $R_{\ell,\ell}=\ell-1$, for $\ell \ge 1$.
\end{proof}

Somewhat surprisingly, the solution of the initial value problem in Proposition~\ref{MISOpropMellzv} can be stated explicitly and implies the following preliminary result.

\begin{prop}
\label{MISOExplizit}
Let $f(z,v)$ and $g(z,v)$ defined as in \eqref{MISOweak1}.
Then the generating functions $M_{\ell}(z,v)$ are for $\ell \ge 1$ given by the following explicit expressions:
\begin{equation}\label{MISOexplicit}
\begin{split}
M_{1}(z,v) & = e^{\int_{0}^{z} f(t,v)dt}, \\
M_{\ell}(z,v) & = v^{\ell-1} (\ell-1)! \big(M_{1}(z,v)\big)^{\ell} = v^{\ell-1} (\ell-1)! \exp\Big(\int_0^{z}\ell f(t,v)dt\Big), \quad \text{for $\ell \ge 2$}.
\end{split}
\end{equation}
\end{prop}
\begin{proof}
First one can check easily that the given generating functions indeed satisfy the initial conditions as stated in Proposition~\ref{MISOpropMellzv}, i.e., $M_{\ell}(0,v) = v^{\ell-1} (\ell-1)!$, $\ell \ge 1$.
To show that the functions stated also satisfy the differential equation~\eqref{MISOode1} we use induction.
Plugging $\ell=1$ into \eqref{MISOode1} it simplifies to 
\begin{equation*}
  \dzp M_{1}(z,v) - f(z,v) M_{1}(z,v)  = 0,
\end{equation*}
which is satisfied by formula~\eqref{MISOexplicit}.
Now let us assume that \eqref{MISOexplicit} holds for all $1 \le r < \ell$.
Then, the expression $b_{\ell}(z,v)$ defined in \eqref{MISObellzv} simplifies as follows:
\begin{align*}
  b_{\ell}(z,v) & = v \sum_{r=1}^{\ell-1} \left[\binom{\ell-1}{r} + \binom{\ell-1}{r-2}\right] M_{r}(z,v) M_{\ell-r}(z,v) \\
  & = v^{\ell-1} (M_{1}(z,v))^{\ell} \cdot \sum_{r=1}^{\ell-1} \left[\binom{\ell-1}{r} + \binom{\ell-1}{r-2}\right] (r-1)! (\ell-r-1)! \\
  & = v^{\ell-1} (\ell-1)! (M_{1}(z,v))^{\ell} \cdot \sum_{r=1}^{\ell-1} \left(\frac{1}{r} + \frac{r-1}{(\ell-r+1)(\ell-r)}\right).
\end{align*}
The sum can be evaluated easily:
\begin{align*}
  & \sum_{r=1}^{\ell-1} \left(\frac{1}{r} + \frac{r-1}{(\ell-r+1)(\ell-r)}\right) = \sum_{r=1}^{\ell-1} \frac{1}{r} + \sum_{r=1}^{\ell-1} \frac{\ell-r-1}{(r+1)r}
  = \sum_{r=1}^{l-1}\frac{\ell}{(r+1)r} \\
  & \quad = \ell \sum_{r=1}^{\ell-1} \left(\frac{1}{r} - \frac{1}{r+1}\right) = \ell \big(1-\frac{1}{\ell}\big) = \ell-1,
\end{align*}
which yields
\begin{equation}
\label{MISOformulabellzv}
  b_{\ell}(z,v) = (\ell-1) v^{\ell-1} (\ell-1)! (M_{1}(z,v))^{\ell}.
\end{equation}
Plugging the function $v^{\ell-1} (\ell-1)! (M_{1}(z,v))^{\ell}$ stated in \eqref{MISOexplicit} into the left hand side of the differential equation \eqref{MISOode1}
one gets after straightforward computations
\begin{align*}
  & \dzp M_{\ell}(z,v) - (\ell f(z,v)-(\ell-1)g(z,v))M_{\ell}(z,v) \\
  & \enspace\enspace = v^{\ell-1} (\ell-1)! \ell (M_{1}(z,v))^{\ell-1} \dzp M_{1}(z,v)
  - (\ell f(z,v) - (\ell-1) g(z,v)) v^{\ell-1} (\ell-1)! (M_{1}(z,v))^{\ell} \\
  & \enspace\enspace = g(z,v) \cdot (\ell-1) v^{\ell-1} (\ell-1)! (M_{1}(z,v))^{\ell},
\end{align*}
which, due to \eqref{MISOformulabellzv}, matches with the right hand side $g(z,v) b_{\ell}(z,v)$ of \eqref{MISOode1}, i.e., formula~\eqref{MISOexplicit} also holds for the value $\ell$.
\end{proof}

\subsection{Establishing a weak limit law by exploiting the singular structure}
As mentioned previously, the so-called method of moments, which will be applied for the analysis of the random variables $\rvL$ and $\rvY$ in the succeeding sections, is not suited for the derivation of a non-degenerate limit law of $\rvR$. Instead, we will determine the singular structure of the generating function $M_{\ell}(z,v)$, as obtained in the previous subsection. Then, we use complex analytic methods to determine the asymptotic behaviour of the $n$-th coefficients of $M_{\ell}(z,v)$, and to obtain a weak limit law. In order to do so, we will build on earlier results concerning the case $\ell=1$: one should here give much credit to Drmota et al.~\cite{Drmota2009}, who have established this case. Since by
\[
M_{\ell}(z,v)=v^{\ell-1}(\ell-1)! \big(M_1(z,v)\big)^{\ell}, \quad \ell\ge 2,
\]
we can relate the analysis of $M_{\ell}(z,v)$ to the corresponding analysis of the special case $\ell=1$, following
closely the arguments of~\cite{Drmota2009}. We start by determining the singularities of $$f(z,v)=\frac{v\lo{z}}{(1-z)v \lo{z} + z(1-v)},$$ as defined in~\eqref{MISOweak1}: the function is singular at $z=1$ due to the logarithmic factor. However, it was observed that the function has another singularity $z_0(v)$, which coincides with $z=1$ for $v=1$. We collect a result of~\cite{Drmota2009}.
\begin{lemma}
\label{MISOweakLemma2}
Set $v=e^w$ and suppose that $|\arg(w)|\le \pi-\delta$ and $0<|w|<\eta$ for some $\delta>0$ and some sufficiently small $\eta>0$. Then, for every $w$ in that range there is exactly one zero of the mapping
$z\mapsto \frac{1}{f(z,e^w)}$, that is asymptotically given by
\[
z_0(e^w)=1-\frac{w}{\log \frac1w}+\frac{w\log\log \frac1w}{\log^2\frac1w}+\mathcal{O}\bigg(\frac{w\big(\log\log \frac1w\big)^2}{\log^3\frac1w} \bigg),
\]
uniformly as $w\to 0$ and $|\arg(w)|\le \pi-\delta$.
\end{lemma}
By Proposition~\ref{MISOExplizit} one can obtain an asymptotic expansion of $M_1(z,v)$ and thus of $M_\ell(z,v)$ using the singularity structure of $f(z,v)$. We state the following important result of~\cite{Drmota2009}:
\begin{lemma}
\label{MISOweakLemma3}
Let $v=v(n)=e^{i\lambda n^{-1}(\log n)^A}$ for a real number $\lambda\neq 0$ and some $A>1$. Then, 
\[
M_1(1,v)=e^{\log n-(A-1)\log\log n+\mathcal{O}(1)},
\]
and 
\begin{equation}
\begin{split}
\label{MISOweakLemma3P1}
M_1(z,v)&= M_1(1,v)\cdot \exp\bigg(-\frac{(1-z)\log(1-z)-(1-z)}{i\lambda n^{-1}(\log n)^A}+\mathcal{O}\Big(\frac{|1-z|^2(\log(|1-z|)^2)}{n^{-2}(\log n)^{2A}}\Big)\\
&\qquad +\mathcal{O}\Big(|1-z|\log(|1-z|)\Big)\bigg),
\end{split}
\end{equation}
uniformly for $|z-1|\le n^{-1}(\log n)^{A-1}$, and $t\in\Delta$. Moreover, if  $|z-z_0(v)|\le n^{-1}(\log n)^{A-1}$, and $t\in\Delta$, then
\begin{equation}
\label{MISOweakLemma3P2}
M_1(z,v)= \Big(\frac{1}{1-\frac{z}{z_0(v)}}\Big)^{1+\frac1{\log n} + \mathcal{O}\Big(\frac{\log\log n}{(\log n)^2}\Big)}
e^{C-(A-1)\frac{\log\log n}{\log n} +\mathcal{O}\Big(\frac{1}{\log n}\Big)},
\end{equation}
for some constant $C$.
\end{lemma}
Consequently, we directly obtain an expansion of $M_{\ell}(z,v)=v^{\ell-1}(\ell-1)! \big(M_1(z,v)\big)^{\ell}$. The next step is to use a Cauchy integral to extract coefficients of $M_\ell(z,v)$, 
and to obtain an asymptotic expansion of the probability generating function of $\rvR$. 
It is convenient to consider the shifted random variable $\hat{R}_{n,\ell}=\rvR-(\ell-1)$.
Following~\cite{Drmota2009}, we use a contour integral as depicted in~\eqref{MISOpic2}, surrounding the singularities $z=1$ and $z=z_0(s)$ 
with winding number one around the origin
\begin{figure}[!htb]
\centering
\includegraphics[scale=0.6]{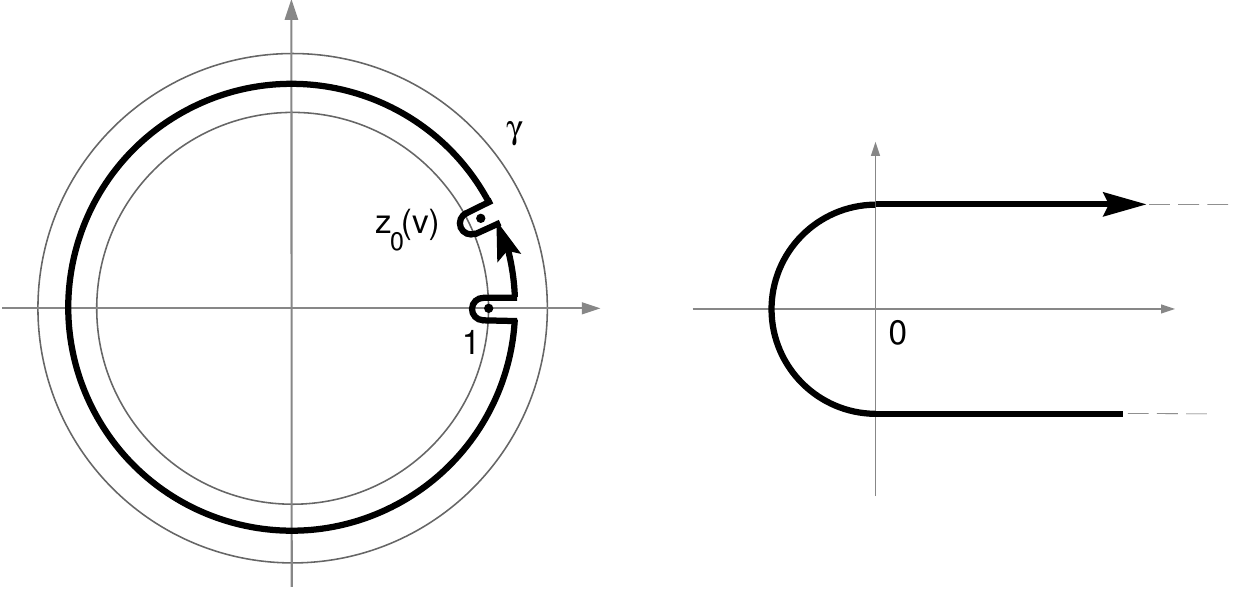}
\caption{Integration path $\gamma$; Hankel contour $\Hank$}
\label{MISOpic2}
\end{figure}

and extract coefficients according to the definition of $M_{\ell}(z,v)$ given in~\eqref{MISOgen1}:
\begin{equation}
\begin{split}
\label{MISOasympt1}
\E(v^{\hat{R}_{n,\ell}})&=\frac{1}{v^{\ell-1}\fallfak{(n-1)}{\ell-1}}[z^{n-\ell}]M_{\ell}(z,v)=
\frac{1}{v^{\ell-1}\fallfak{(n-1)}{\ell-1}2\pi i}\int_\gamma \frac{M_\ell(z,v)}{z^{n-\ell+1}}dz\\
&=\frac{(\ell-1)!}{\fallfak{(n-1)}{\ell-1}}\frac{1}{2\pi i}\int_\gamma \frac{ \big(M_1(z,v)\big)^{\ell}}{z^{n-\ell+1}}dz.
\end{split}
\end{equation}
Note that we assume that $v=v(n)=e^{i\lambda n^{-1}(\log n)^{2}}$ for $\lambda\in\R\setminus\{0\}$.
We obtain the following result.
\begin{lemma}
\label{MISOweakLemma4}
Assume that $v=v(n)=e^{i\lambda n^{-1}(\log n)^2}$ for $\lambda \in\R\setminus\{0\}$. Then
\[
\E(v^{\hat{R}_{n,\ell}})=z_0(v)^{-n}\Big(1+\mathcal{O}\Big(\frac
{1}{\log n}\Big)\Big),
\]
where $z_0(v)\neq 1$ is a zero of the function $\frac{1}{f(z,v)}$, with $f(z,v)$ given in~\eqref{MISOweak1}, 
and $z_0(v)$ satisfying the asymptotic expansion
\[
z_0(v)=1-\frac{i\lambda \log n}{n}-\frac{i\lambda \log\log n}{n}-\frac{i\lambda\log(i\lambda)}n + \mathcal{O}
\Big(\frac{(\log\log n)^2}{n\log n}\Big).
\]
\end{lemma}
\begin{proof}
We note first that the expansion of $z_0(v)$ follows rather quickly from Lemma~\ref{MISOweakLemma2} with the choice $v=e^{w}=e^{i\lambda n^{-1}(\log n)^2}$.
Next we turn to the curve $\gamma$. It consists of four parts, $$\gamma=\gamma_1\cup \gamma_{2}\cup \gamma_{3}\cup\gamma_{4}.$$
Two so-called Hankel-contours $\gamma_{2},\gamma_{4}$, surrounding the singularities $z=1$ and $z=z_0(s)$, and the remaining paths $\gamma_1$ and $\gamma_3$ stem from a circle of radius $r>1$;
in particular we use $r=1+\frac{\log n}{n}$. Let $\Hank$ denote the major part of a Hankel contour, consisting of a half circle of radius $1$ and two lines of length $\log n$:
\[
\Hank=\{z\in\C:\,|z|=1, \Re(z)\le 0 \} \cup \{ z\in\C :\, 0\le \Re(z)\le \log n,\quad \Im(z)=\pm 1\}.
\]
We consider first the integral $I_1$ around the singularity $z=1$, 
using the substitution $z=1+\frac{u}{n}$, with $u\in \Hank$.
We have
\[
I_1=\frac{(\ell-1)!}{\fallfak{(n-1)}{\ell-1}}\frac{1}{2\pi i}\int_{\gamma_{2}} \frac{ \big(M_1(z,v)\big)^{\ell}}{z^{n-\ell+1}}dz
=\frac{(\ell-1)!}{\fallfak{n}{\ell}}\frac{1}{2\pi i}\int_{\Hank} \frac{ \Big(M_1\big(1+\frac{u}{n},v\big)\Big)^{\ell}}{(1+\frac{u}{n})^{n-\ell+1}}du.
\]
By Lemma~\eqref{MISOweakLemma3} equation~\ref{MISOweakLemma3P1} we obtain that
\[
\Big|M_1(1+\frac{u}{n},v)\Big|=\Gro\Big(\frac{n}{\log n}\Big),
\]
for $u\in\Hank$, such that 
\[
\frac{(\ell-1)!}{\fallfak{n}{\ell}}\Big|M_1(1+\frac{u}{n},v)\Big|^\ell=\Gro\Big(\frac{1}{\log^{\ell} n}\Big).
\]
Moreover, we get
\[
(1+\frac{u}{n})^{-n+\ell-1}=e^{-u}\Big(1+\Gro\big(\frac{|u|^2}{n}\big)\Big).
\]
Consequently, decomposing the Hankel counter into the half circle and the two rays of length $\log n$ we get
\[
I_1=\Gro\bigg(\frac{1}{\log^{\ell} n} + \frac{1}{\log^{\ell} n} \int_{0}^{\log n}e^{-u}du\bigg)=\Gro\bigg(\frac{1}{\log^{\ell} n}\bigg).
\]

\smallskip

Next we consider the main contribution - the integral around the singularity $z=z_0(v)$.
We use the substitution $z=z_0(v)(1+\frac{u}{n})$, with $u\in \Hank$.
\[
I_2=\frac{(\ell-1)!}{\fallfak{(n-1)}{\ell-1}}\frac{1}{2\pi i}\int_{\gamma_{2}} \frac{ \big(M_1(z,v)\big)^{\ell}}{z^{n-\ell+1}}dz
=\frac{(\ell-1)!z_0(v)^{-n+\ell}}{\fallfak{n}{\ell}}\frac{1}{2\pi i}\int_{\Hank} \frac{ \Big(M_1\big(z_0(v)(1+\frac{u}{n}),v\big)\Big)^{\ell}}{(1+\frac{u}{n})^{n-\ell+1}}du.
\]
By Lemma~\ref{MISOweakLemma3}, equation~\eqref{MISOweakLemma3P2}, and using continuity arguments implying that $C=-1$ (compare with~\cite{Drmota2009}), it can be shown that for $u\in\Hank$ 
\begin{equation*}
\begin{split}
M_1(1+\frac{u}{n},v)&=n^{1+\frac{1}{\log n}+\Gro(\frac{\log\log n}{(\log n)^2})} (-u)^{-1-\frac{1}{\log n} +\Gro(\frac{\log\log n}{(\log n)^2})}e^{-1+\Gro(\frac{\log\log n}{\log n})}\\
&=n^{1+\Gro(\frac{\log\log n}{(\log n)^2})} (-u)^{-1-\frac{1}{\log n} +\Gro(\frac{\log\log n}{(\log n)^2})}e^{\Gro(\frac{\log\log n}{\log n})}.
\end{split}
\end{equation*}
Consequently, we obtain
\begin{equation*}
\begin{split}
I_2&=\frac{n^{\ell}(\ell-1)!z_0(v)^{-n+\ell}}{\fallfak{n}{\ell}}\frac{1}{2\pi i}\int_{\Hank} \frac{ \Big(M_1\big(z_0(v)(1+\frac{u}{n}),v\big)\Big)^{\ell}}{(1+\frac{u}{n})^{n-\ell+1}}du\\
&=\frac{n^{\ell}(\ell-1)!z_0(v)^{-n+\ell}}{\fallfak{n}{\ell}}\frac{1}{2\pi i}\int_{\Hank} (-u)^{-\ell-\frac{\ell}{\log n} }e^{-u}du\Big(1+\Gro(\frac{1}{\log n})\Big).
\end{split}
\end{equation*}
Using the contour integral representation by Hankel of the reciprocal of the gamma function 
\[
\frac{1}{2\pi i}\int_\Hank (-u)^{-s}e^{-u}=\frac{1}{\Gamma(s)} + \Gro\Big(\frac{1}{n(\log n)^{\Re(s)}}\Big),
\]
we get
\[
I_2=\frac{n^{\ell}(\ell-1)!z_0(v)^{-n+\ell}}{\Gamma(\ell+\frac{\ell}{\log n})\fallfak{n}{\ell}}\Big(1+\Gro(\frac{1}{\log n})\Big).
\]
Expansion of $\frac{n^{\ell}}{\fallfak{n}{\ell}}$, $z_0(v)^{\ell}$ and $\Gamma(\ell+\frac{\ell}{\log n})$ leads to
\[
I_2=z_0(v)^{-n}\Big(1+\Gro(\frac{1}{\log n})\Big).
\]

Next we consider the integral for $|z-1|\ge \epsilon$. Since $M_1(z,v)$ and thus also $M_\ell(z,v)$ has no singularities except $z=1$ and $z=z_0(v)$ the function is 
uniformly bounded for $|z|=R$ and $|z-1|\ge \epsilon$. Consequently, the integral $I_3$ satisfies
\[
I_3=\frac{(\ell-1)!}{\fallfak{(n-1)}{\ell-1}}\frac{1}{2\pi i}\int_{|z|=R, |z-1|\ge \epsilon} \frac{ \big(M_1(z,v)\big)^{\ell}}{z^{n-\ell+1}}dz
= \Gro\Big(\frac{1}{R^{n-\ell+1}n^{\ell-1}}\Big)
= \Gro\Big(\frac{1}{n^{\ell}}\Big).
\]

It is known~\cite{Drmota2009} that in the remaining case $|z=R|$, $|z-1|<\epsilon$, it holds
$$M_1(z,v)=\Gro\Big(\frac{n}{\log n}\Big).$$ 
Hence, 
\[
I_4=\frac{(\ell-1)!}{\fallfak{(n-1)}{\ell-1}}\frac{1}{2\pi i}\int_{|z|=R, |z-1|< \epsilon} \frac{ \big(M_1(z,v)\big)^{\ell}}{z^{n-\ell+1}}dz
= \Gro\Big(\frac{n^\ell}{\log^\ell( n) \, n^{\ell-1} R^{n-\ell+1}}\Big)
= \Gro\Big(\frac{1}{\log^\ell(n)} \Big).
\]
\end{proof}

\bigskip

In order to obtain the weak limit from Lemma~\ref{MISOweakLemma4} we consider the 
shifted and normalized random variable 
\[
R^{\ast}_{n,\ell}=\frac{\rvR-(\ell-1)-\frac{n}{\log n} - \frac{n\log\log n}{(\log n)	^2}}{\frac{n}{(\log n)^2}}=\frac{\rvhR-a_n}{b_n},
\]
with 
\[
a_n=\frac{n}{\log n} + \frac{n\log\log n}{(\log n)	^2},\qquad b_n=\frac{n}{(\log n)^2}.
\]
The characteristic function of $R^*_{n,\ell}$
is given by
\[
\E(e^{i\lambda R^*_{n,\ell}})=\E(e^{i\lambda (\rvhR-a_n)/b_n})= \E(e^{i\lambda \rvhR/b_n})e^{-i\lambda a_n/b_n}.
\]
Note that
\[
e^{-i\lambda a_n/b_n}= e^{-i\lambda\log n-i\lambda\log\log n}.
\]
By Lemma~\ref{MISOweakLemma4} we obtain for the characteristic function - the probability generating function of $\rvhR$ with $v=v(n)=e^{i\lambda n^{-1}(\log n)^2}= e^{i\lambda/ b_n}$ - the result
\begin{equation}
\begin{split}
\E(e^{i\lambda \rvhR /b_n})&=\E\big((v(n))^{\rvhR}\big)
= e^{-n \log z_0(v)}\Big(1+\mathcal{O}\Big(\frac1{\log n}\Big)\Big)\\
&= e^{i\lambda \log n+ i\lambda \log\log n+i\lambda\log(i\lambda)} + \mathcal{O}\Big(\frac{(\log\log n)^2}{\log n}\Big).
\end{split}
\end{equation}
Consequently, 
\[
\E(e^{i\lambda R^*_{n,\ell}})= e^{i\lambda\log(i\lambda)} + \mathcal{O}\Big(\frac{(\log\log n)^2}{\log n}\Big).
\]
Finally, noting that
\[
i\lambda\log(i\lambda)=i\lambda( \log |\lambda| + i\sgn(\lambda)\frac{\pi}2)= i\lambda\log |\lambda|-\frac{\pi}2\sgn(\lambda)\lambda
=i\lambda\log |\lambda|-\frac{\pi}2|\lambda|,
\]
proves that 
\[
\E(e^{i\lambda R^*_{n,\ell}})= e^{i\lambda\log |\lambda|-\frac{\pi}2|\lambda|} + \mathcal{O}\Big(\frac{(\log\log n)^2}{\log n}\Big).
\]
This implies that the characteristic function of $R^*_{n,\ell}$ converges to the characteristic function of a stable random variable with characteristic quadruple 
$(0,\frac{\pi}2,1,-1)$.

\section{Isolating the nodes \texorpdfstring{$n+1-l,\dots,n$}{n+1-l,...n}\label{MISOSecEnnPlusEinsMinusEll}}

\subsection{Generating functions description}
In order to study the r.v.\ $\rvL$ satisfying the distributional recurrence \eqref{MISOdistributionalL1}, we introduce for $\ell \ge 1$ the generating functions
\begin{equation}
\label{MISOgen2}
   N_{\ell}(z,v) := \sum_{n \ge \ell} (n-1)^{\underline{\ell-2}} \, \E(v^{\rvL})z^{n+1-\ell};
\end{equation}
note that in the special case $\ell=1$ one has $\fallfak{(n-1)}{-1} = \frac{1}{n}$. The
starting point of our considerations is the following Proposition.


\begin{prop}
The generating functions $N_{\ell}(z,v) $ satisfy for $\ell \ge 1$ the following second order differential equations:
\begin{equation}
    \label{MISOode6}
   \dzpZwo N_\ell(z,v) +(\ell-1)g(z,v) \dzp N_{\ell}(z,v) %
        - \frac{vg(z,v)}{1-z} N_\ell(z,v)%
   = g(z,v)\cdot b_{\ell}(z,v),
\end{equation}
with $g(z,v)$ as defined in \eqref{MISOweak1} and where $b_{\ell}(z,v)$ is given by
\begin{equation}\label{MISObellzvDef}
    b_{\ell}(z,v) = v \sum_{r=1}^{\ell-1}\binom{\ell}{r}  N_r(z,v) \dzpZwo N_{\ell-r}(z,v)
        + \sum_{r=1}^{\ell-1}\binom{\ell}{r-1}(\ell-r-1)! v^{\ell-r} \dzp N_r(z,v).
\end{equation}
\end{prop}
\begin{remark}
In contrast to the previous section studying $M_{\ell}(z,v)$, so far we are not able to derive a closed form expression for $N_\ell(z,v)$, not even in the simplest case $\ell=1$. 
\end{remark}
\begin{proof}
We obtain from the distributional equation~\eqref{MISOdistributionalL1} the following recurrence for the probability generating function $\E(v^{\rvL})$:
\begin{equation}
\label{MISOrec2}
\begin{split}
     \E(v^{\rvL}) =& v\sum_{r=1}^{\ell} \sum_{k=r}^{n-1}\P\{\hat{I}_n=k,\hat{J}_\ell=r\}
     \E(v^{L_{k, r}})\E(v^{L_{n-k, \ell-r}}), \quad \text{for $1 \le \ell \le n$ and $n \ge 2$},
     \end{split}
\end{equation}
with initial value $\E(L_{1, 1})= 1$. 
When translating the recurrence relation~\eqref{MISOrec2} into a differential equation,
for $\ell\geq 4$ we always have to distinguish between the four cases
$r=1$, $1< r < \ell-1$, $r=\ell-1$ and $r=\ell$.
Moreover, note that  $\fallfak{(n-1)}{\ell-1}/\fallfak{(n-1)}{\ell-r} =
\fallfak{(n+r-\ell-1)}{r-1}$. Multiplying \eqref{MISOrec2} with $(n-1)\fallfak{(n-1)}{\ell-1}z^{n-l}$ and taking summation over all $n\geq \ell$
leads then to a second order differential equation for
$N_{\ell}(z,v)$, where the functions $N_{r}(z,v)$ with $r < \ell$ are
appearing in the inhomogeneous part. Again we do not carry out these straightforward computations, which eventually give
\begin{equation}\label{MISONellzvODE}
  \big((1-z)v \log\big(\frac{1}{1-z}\big)+z(1-v)\big) \dzpZwo N_{\ell}(z,v) + (\ell-1) \dzp N_{\ell}(z,v) - \frac{v}{1-z} N_{\ell}(z,v) = b_{\ell}(z,v),
\end{equation}
with
\begin{align*}
  b_{\ell}(z,v) & = v \sum_{r=1}^{\ell-1}\binom{\ell-1}{r-1}%
    \Big( N_r(z,v) \dzpZwo N_{\ell-r}(z,v) +  N_{\ell-r}(z,v)\dzpZwo N_r(z,v) \Big)\\
    & \quad \mbox{} + \sum_{r=1}^{\ell-1}\binom{\ell}{r-1}(\ell-r-1)!  v^{\ell-r} \dzp N_r(z,v)\\
    & = v \sum_{r=1}^{\ell-1}\binom{\ell}{r}  N_r(z,v) \dzpZwo N_{\ell-r}(z,v) + \sum_{r=1}^{\ell-1}\binom{\ell}{r-1}(\ell-r-1)! v^{\ell-r} \dzp N_r(z,v),
\end{align*}
and thus show the stated result.
\end{proof}

\subsection{Asymptotics of the moments}
In order to prove the limit law for $\rvL$ we will use the so-called method of moments, i.e., we will show that the $s$-th positive integer moments of $\rvL$ converge, after suitable normalization, to the corresponding moments of a beta-distributed r.v.\ and apply the Fr\'echet-Shohat moment convergence theorem~~\cite{Loe1977}. Together with the fact that the Beta-distribution is uniquely determined by its $s$-th integer moments, this will imply Theorem~\ref{MISOthe2}.

In order to get the moments of $\rvL$ we introduce, for $\ell \ge 1$ and $s \ge 0$, the functions
\begin{equation}\label{MISONellzvDef1}
  N_{\ell,s}(z) := E_v D_v^s N_\ell(z,v) = \sum_{n \ge \ell} (n-1)^{\underline{\ell-2}} \, \E\big(\rvL^{\underline{s}}\big) z^{n+1-\ell},
\end{equation}
with $N_{\ell}(z,v)$ defined in \eqref{MISOgen2}. Then we can determine the $s$-th factorial moments of $\rvL$ simply via
\begin{equation}\label{MISONellzvDef2}
  \E(\fallfak{\rvL}s) = \frac1{(n-1)^{\underline{\ell-2}}}[z^{n+1-\ell}]N_{\ell,s}(z).
\end{equation}


To deduce the asymptotic behaviour of the coefficients of the functions $N_{\ell,s}(z)$ (and thus of the factorial moments $\E(\fallfak{\rvL}s)$) we will determine the local behaviour of $N_{\ell,s}(z)$ around their unique dominant singularity $z=1$ (as we shall see later on) and apply singularity analysis~\cite{SingAna1990}. In order to apply singularity analysis (i.e., transfer lemmata which allow to ``translate'' the local behaviour of a generating function around its dominant singularity into an asymptotic growth behaviour of the coefficients) it is necessary that the functions involved are analytic in a domain larger than the circle of convergence, namely, the functions have to be analytic for indented discs $\Delta :=\Delta(\phi,\eta) = \{z:\,\, |z| < 1 +\eta,\,\, |\Arg(z-1)|> \phi\}$, with $\eta>0$, $0<\phi<\frac{\pi}{2}$. Such functions are called $\Delta$-regular (see \cite{FlaFillKap2005}). Here, we have restricted ourselves to a definition of $\Delta$-regularity for functions with unique dominant singularity $z=1$, since this is sufficient for our purpose.
Later on we will show inductively that all the functions $N_{\ell,s}(z)$ are $\Delta$-regular, since they are generated from $\Delta$-regular functions via basic arithmetical operations together with the operations differentiation and integration.
In this context we require the following lemma, which is a slight generalization of corresponding ones shown in \cite{FlaFillKap2005} and that can be obtained in a completely analogeous manner; thus we omit here the proof.
\begin{lemma}[Singular differentiation and integration]\label{MISOsingana}
Let $f(z)$ be a $\Delta$-regular function, an analytic function in the domain $\Delta
:=\Delta(\phi,\eta)$,
\begin{equation*}
  \Delta(\phi,\eta) = \{ z:\,\, |z| < 1 +\eta,\,\, |\Arg(z-1)|> \phi\},%
\end{equation*}
with $\eta>0$, $0<\phi<\frac{\pi}{2}$, satisfying for $z\to 1$ the
expansion
\begin{equation*}
    f(z) = \Gro\Big(\frac{1}{(1-z)^{a}\Lo{z}{b}}\Big)
\end{equation*}
for $a>1$ and $b\ge 1$. Then $\int_{0}^{z} f(t) dt$ and $f'(z)$
are also $\Delta$-regular and they admit the expansions
\begin{equation*}
    \int_{0}^{z} f(t) dt =
    \Gro\Big(\frac{1}{(1-z)^{a-1}\Lo{z}{b}}\Big),\quad\text{and}\quad
    f'(z)  = \Gro\Big(\frac{1}{(1-z)^{a+1}\Lo{z}{b}}\Big).
\end{equation*}
\end{lemma}

The next lemma states the analytic properties of the functions $N_{\ell,s}(z)$, which turn out to be crucial to the approach presented.
\begin{lemma}
\label{LemmaMomsEnEll}
The generating functions $N_{\ell,s}(z)= E_v D_v^s N_\ell(z,v)$ are, for all $\ell \ge 1$ and $s \ge 0$, $\Delta$-regular functions.
Moreover, for $\ell, s \ge 1$, $N_{\ell,s}(z)$ admits the following local expansion around the dominant singularity $z=1$:
\begin{equation}\label{MISONellsexp}
    N_{\ell,s}(z) = \frac{\ell(\ell+s-2)!}{(\ell+s)(1-z)^{\ell+s-1}\Lo{z}{s}} +  \Gro\Big(\frac{1}{(1-z)^{\ell+s-1}\Lo{z}{s+1}}\Big).
\end{equation}
\end{lemma}
\begin{proof}
To prove Lemma~\ref{LemmaMomsEnEll} we will use induction with respect to $\ell$ and $s$.
First we consider the case $s=0$. Using definition~\eqref{MISONellzvDef1} we obtain that the functions $N_{\ell,0}(z) = \sum_{n \ge \ell} (n-1)^{\underline{\ell-2}} z^{n+1-\ell}$ are given by
\begin{equation*}
  N_{\ell,0}(z) = \begin{cases} \lo{z}, & \quad \ell=1, \\ \frac{(\ell-2)!}{(1-z)^{\ell-1}} - (\ell-2)!, & \quad \ell \ge 2. \end{cases} 
\end{equation*}
Thus the functions $N_{\ell,0}(z)$, $\ell \ge 1$, are $\Delta$-regular. Note that for $\ell \ge 2$ they even admit the local expansion~\eqref{MISONellsexp}.

Next we consider the case $\ell, s \ge 1$ and start with the differential equation~\eqref{MISONellzvODE} for the functions $N_{\ell}(z,v)$ with inhomogeneous part $b_{\ell}(z,v)$ given by \eqref{MISObellzvDef}. Applying the operator $E_{v} D_{v}^{s}$ to this equation yields the following second order differential equation for $N_{\ell,s}(z)$, $\ell, s \ge 1$:
\begin{equation}\label{MISONellODE}
  (1-z) \lo{z} N_{\ell,s}''(z) + (\ell-1) N_{\ell,s}'(z) - \frac{1}{1-z} N_{\ell,s}(z) = b_{\ell,s}(z),
\end{equation}
where the inhomogeneous part $b_{\ell,s}(z)$ is given by
\begin{equation}
\begin{split}\label{MISObellDef}
b_{\ell,s}(z) & = s \Big(1 - (1-z) - (1-z)\lo{z}\Big) N_{\ell,s-1}''(z) + \frac{s}{1-z} N_{\ell,s-1}(z) \\
& \quad \mbox{} + \sum_{r=1}^{\ell-1} \binom{\ell}{r} \sum_{j=0}^{s} N_{r,j}(z) N_{\ell-r,s-j}''(z) + s \sum_{r=1}^{\ell-1} \binom{\ell}{r} \sum_{j=0}^{s-1} N_{r,j}(z) N_{\ell-r,s-1-j}''(z) \\
& \quad \mbox{} + \sum_{r=1}^{\ell-1} \binom{\ell}{r-1} (\ell-r-1)! \sum_{j=0}^{s} \binom{s}{j} (\ell-r)^{\underline{j}} N_{r,s-j}'(z).
\end{split}
\end{equation}


As one can check easily, the homogeneous differential equation corresponding to \eqref{MISONellODE} has the general solution
\begin{equation}
\begin{split}
    N_{\ell,s}^{[h]}(z) & = C_{1,s} \cdot  N_{\ell}^{[h_1]}(z) + C_{2,s} \cdot N_{\ell,s}^{[h_2]}(z),
\end{split}
\end{equation}
with solutions  $N_{\ell}^{[h_1]}(z)$, $N_{\ell}^{[h_2]}(z)$ given by
\begin{equation}
\begin{split}
\label{MISOdefSols}
N_{\ell}^{[h_1]}(z) & = \ell-1 + \lo{z},\\
N_{\ell}^{[h_2]}(z)& = \big(\ell-1 + \lo{z}\big) \int\frac{dz}{\Lo{z}{\ell-1} \cdot \big(\ell-1 + \lo{z}\big)^2}.
\end{split}
\end{equation}
Note that given a function $f(z)$, we assume in the definition of an antiderivative $\int f(z) dz := \int_{\alpha}^{z} f(t) dt$, with a real $0 < \alpha < 1$.

Applying the variation of parameters-method leads then to the following particular solution of the inhomogeneous differential equation \eqref{MISONellODE}:
\begin{equation}
\label{MISOpartsol}
    N_{\ell,s}^{[p]}(z) = N_{\ell}^{[h_1]}(z) \int \frac{-B_{\ell,s}(z) N_{\ell}^{[h_2]}(z)}{D_\ell(z)} dz +
        N_{\ell}^{[h_2]}(z) \int \frac{B_{\ell,s}(z) N_{\ell}^{[h_1]}(z)}{D_\ell(z)} dz,
\end{equation}
where $B_{\ell,s}(z) = \frac{b_{\ell,s}(z)}{(1-z) \lo{z}}$ and
\begin{equation}
D_\ell(z) = N_{\ell}^{[h_1]}(z) \dzp N_{\ell}^{[h_2]}(z) - N_{\ell}^{[h_2]}(z) \dzp N_{\ell}^{[h_1]}(z) = \frac{1}{\Lo{z}{\ell-1}}
\end{equation}
is the Wronski determinant of the two homogeneous solutions $N_{\ell}^{[h_1]}(z)$ and $N_{\ell}^{[h_2]}(z)$.
Combining the expressions appearing in \eqref{MISOpartsol} allows to adapt the limit of integration to $\alpha=0$ yielding the following particular solution of \eqref{MISONellODE}, which, as discussed below, turns out to satisfy also the initial conditions, i.e., which is the required solution $N_{\ell,s}(z)$:
\begin{multline}\label{MISONellssol}
  N_{\ell,s}(z) = \big(\ell-1+\lo{z}\big) \int_{0}^{z} \bigg(\int_{0}^{t} \frac{\Lo{u}{\ell-2} \cdot \big(\ell-1+\lo{u}\big) b_{\ell,s}(u)}{1-u} du\bigg) \\
  \times \frac{dt}{\Lo{t}{\ell-1} \cdot \big(\ell-1+\lo{t}\big)^{2}},
\end{multline}
with $b_{\ell,s}(z)$ defined in \eqref{MISObellDef}. Note that according to \eqref{MISONellzvDef1} the initial conditions are given by $N_{\ell,s}(0)=0$ and
$N_{\ell,s}'(0)=(\ell-1)! \E\big(L_{\ell,\ell}^{\underline{s}}\big)$. Since $L_{1,1}=0$ we get in particular $N_{1,s}'(0) = 0$, for $s \ge 1$. Taking into account $N_{\ell,s}(0)=0$ and considering \eqref{MISONellODE} we further obtain the relation $N_{\ell,s}'(0) = \frac{1}{\ell-1} b_{\ell,s}(0)$, for $\ell \ge 2$, which implies $b_{\ell,s}(0) = (\ell-2)! \E\big(L_{\ell,\ell}^{\underline{s}}\big)$. Furthermore, \eqref{MISObellDef} yields $b_{1,s}(0)=0$, for $s \ge 1$, and thus $N_{1,s}'(0)=b_{1,s}(0)$. But this initial conditions
(with $\ell, s \ge 1$):
\begin{equation*}
  N_{\ell,s}(0)=0, \quad N_{\ell,s}'(0) = \begin{cases} \frac{1}{\ell-1} b_{\ell,s}(0), & \ell \ge 2, \\ b_{1,s}(0), & \ell=1, \end{cases}
\end{equation*}
are exactly the ones satisfied by the given solution \eqref{MISONellssol}.

We observe that the representation \eqref{MISONellssol} together with the closure properties for singular differentiation and integration inductively shows that all $N_{\ell,s}(z)$, $\ell,s \ge 1$, are $\Delta$-regular functions.

It remains to show in an inductive way the local expansions \eqref{MISONellsexp} of $N_{\ell,s}(z)$ in a complex neighbourhood of $z=1$. We will first consider the case $\ell=1$, with an arbitrary $s \ge 1$, and assume that \eqref{MISONellsexp} holds for $\ell=1$ and all $1 \le r<s$. Plugging $\ell=1$ into \eqref{MISONellssol} yields the representation
\begin{equation}\label{MISON1szsol}
  N_{1,s}(z) = \lo{z} \int_{0}^{z} \left(\int_{0}^{t} \frac{b_{1,s}(u)}{1-u} du\right) \frac{dt}{\Lo{t}{2}},
\end{equation}
with
\begin{equation*}
  b_{1,s}(z) = s \Big(1-(1-z) - (1-z) \lo{z}\Big) N_{1,s-1}''(z) + \frac{1}{1-z} N_{1,s-1}(z).
\end{equation*}
Using the induction hypothesis and Lemma~\ref{MISOsingana} together with the known functions $N_{\ell,0}(z)$ easily shows the local expansion
\begin{equation*}
  b_{1,s}(z) = \frac{s!}{(1-z)^{s+1} \Lo{z}{s-1}} \cdot \Big(1+\mathcal{O}\Big(\frac{1}{\lo{z}}\Big)\Big).
\end{equation*}
Due to \eqref{MISON1szsol}, further applications of Lemma~\ref{MISOsingana} for singular integration lead then to the local expansion
\begin{equation*}
  N_{1,s}(z) = \frac{(s-1)!}{(s+1) (1-z)^{s} \Lo{z}{s}} \cdot \Big(1+\mathcal{O}\Big(\frac{1}{\lo{z}}\Big)\Big),
\end{equation*}
which shows the required result for the case $\ell=1$.

Now we consider the case $\ell \ge 2$, with an arbitrary $s \ge 1$, and assume that \eqref{MISONellsexp} holds for all $(j,r)$ with $1 \le j \le \ell$, $1 \le r \le s$ and $(j,r) \neq (\ell,s)$.
Using the induction hypothesis together with singular differentiation we can examine each summand of $b_{\ell,s}(z)$ as given in \eqref{MISObellDef}. It turns out that the main contribution is coming from the following expressions:
\begin{align*}
  & \bullet \; s \Big(1-(1-z)-(1-z) \lo{z}\Big) N_{\ell,s-1}''(z) \\
  & \quad = \frac{s \ell (\ell+s-1)!}{(\ell+s-1) (1-z)^{\ell+s} \Lo{z}{s-1}} \cdot \Big(1+\mathcal{O}\Big(\frac{1}{\lo{z}}\Big)\Big), \\
  & \bullet \; \sum_{r=1}^{\ell-1} \binom{\ell}{r} \sum_{j=0}^{s} N_{r,j}(z) N_{\ell-r,s-j}''(z) = \ell N_{1,0}(z) N_{\ell-1,s}''(z) + \sum_{\begin{smallmatrix} 1 \le r \le \ell-1, \\ 0 \le j \le s, \; (j,r) \neq (0,1) \end{smallmatrix}} \binom{\ell}{r} N_{r,j}(z) N_{\ell-r,s-j}''(z) \\
  & \quad = \frac{\ell (\ell-1) (\ell+s-1)!}{(\ell+s-1) (1-z)^{\ell+s} \Lo{z}{s-1}} \cdot \Big(1+\mathcal{O}\Big(\frac{1}{\lo{z}}\Big)\Big),
\end{align*}
whereas the contribution of the remaining terms of $b_{\ell,s}(z)$ is of order $\mathcal{O}\left(\frac{1}{(1-z)^{\ell+s} \log^{s}\left(\frac{1}{1-z}\right)}\right)$.
Thus, adding these contributions, we obtain that $b_{\ell,s}(z)$ has the following local expansion around $z=1$:
\begin{equation}\label{MISObellsexp}
  b_{\ell,s}(z) = \frac{\ell (\ell+s-1)!}{(1-z)^{\ell+s} \Lo{z}{s-1}} \cdot \Big(1+\mathcal{O}\Big(\frac{1}{\lo{z}}\Big)\Big).
\end{equation}
Using the representation \eqref{MISONellssol} and expansion \eqref{MISObellsexp}, straightforward applications of singular integration yield the following local expansion of $N_{\ell,s}(z)$:
\begin{equation*}
  N_{\ell,s}(z) = \frac{\ell (\ell+s-2)!}{(\ell+s) (1-z)^{\ell+s-1} \Lo{z}{s}} \cdot \Big(1+\mathcal{O}\Big(\frac{1}{\lo{z}}\Big)\Big),
\end{equation*}
which shows the required result for the case $\ell \ge 2$ and completes the proof of Lemma~\ref{LemmaMomsEnEll}.
\end{proof}

Using Lemma~\ref{LemmaMomsEnEll} and \eqref{MISONellzvDef2} immediately yields, by an application of basic singularity analysis, the following asymptotic growth behaviour of the $s$-th factorial moments of $\rvL$:
\begin{equation}\label{MISOLnellfakmom}
   \E(\fallfak{\rvL}{s}) = \frac{1}{\fallfak{(n-1)}{\ell-2}}[z^{n+1-\ell}]N_{\ell,s}(z)= 
   \frac{\ell n^s}{(\ell+s)\log^s n}  + \mathcal{O}\Big(\frac{n^s}{\log^{s+1}{n}}\Big), \quad \text{for $s, \ell \ge 1$}.
\end{equation}
Since the sequence of $s$-th integer moments of a r.v.\ $X$ can be obtained from the corresponding sequence of $s$-th factorial moments via the relation
\begin{equation}\label{MISOfallfakordmom}
  \E(X^{s}) = \sum_{j=1}^{s} \Stir{s}{j} \E(X^{\underline{j}}), \quad \text{for $s \ge 1$},
\end{equation}
we further get from \eqref{MISOLnellfakmom} the following asymptotic expansion of the $s$-th moments of $\rvL$, which proves the respective part of Theorem~\ref{MISOthe2}:
\begin{equation*}
     \E(\rvL^s) = \sum_{j=1}^{s}\Stir{s}{j}\E(\fallfak{\rvL}{j}) = \Stir{s}{s}\E(\fallfak{\rvL}{s}) + \mathcal{O}\Big(\frac{n^{s-1}}{\log^{s}{n}}\Big)
	 = \frac{\ell n^s}{(\ell+s)\log^s n} + \mathcal{O}\Big(\frac{n^s}{\log^{s+1}{n}}\Big).
\end{equation*}
Thus, after suitable normalization, the $s$-th moments of $\rvL$ converge to the moments of a beta-distributed random variable with parameters $1$ and $\ell$:
\begin{equation*}
\E\Big(\big(\frac{\log n}{n} \rvL\big)^s\Big) =  \frac{\ell}{\ell+s} + \mathcal{O}\Big(\frac{1}{\log n}\Big),
\end{equation*}
which proves the limiting distribution result given in Theorem~\ref{MISOthe2}.

\section{Isolating randomly selected nodes\label{MISOSecRandom}}

\subsection{Generating functions description\label{MISOSecEllRandom}}
Now we study the r.v.\ $\rvY$ satisfying the distributional recurrence \eqref{MISOdistributionalY1}, where we use an approach similar to the one carried out in Section~\ref{MISOSecEnnPlusEinsMinusEll}.
Here we introduce for $\ell \ge 1$ the generating functions
\begin{equation}
\label{MISOGellzvDef}
   G_{\ell}(z,v) := \sum_{n \ge \ell} \frac{\binom{n}{\ell}}{n} \E(v^{\rvY}) z^{n}.
\end{equation}

The following proposition gives a recursive description of the sequence of functions $G_{\ell}(z,v)$.
\begin{prop}\label{MISGellzvProp}
The generating functions $G_{\ell}(z,v) $ satisfy for $\ell \ge 1$ the following second order differential equations:
\begin{multline}
    \label{MISOGellzvODE}
   (1-z) \Big(z(1-v) + v(1-z) \lo{z}\Big) \dzpZwo G_\ell(z,v) \\
   \mbox{} - \Big(z(1-v) + 2v(1-z)\lo{z}\Big) \dzp G_{\ell}(z,v) + (1-v) G_\ell(z,v) 
   = b_{\ell}(z,v),
\end{multline}
where the inhomogeneous part $b_{\ell}(z,v)$ is given by
\begin{equation}\label{MISObellzvDefG}
\begin{split}
    b_{\ell}(z,v) & = - (1-z)^{2} \sum_{r=1}^{\ell-1} \dzpZwo G_{r}(z,v) \cdot G_{\ell-r}(z,v) \\
    & \quad \mbox{} + v (1-z)^{2} \sum_{r=1}^{\ell-1} \sum_{q=1}^{\ell-r-1} \dzp G_{r}(z,v) \cdot \dzp G_{q}(z,v) \cdot G_{\ell-r-q}(z,v) \\
    & \quad \mbox{} + v (1-z)^{2} \lo{z} \sum_{r=1}^{\ell-1} \dzp G_{r}(z,v) \cdot \dzp G_{\ell-r}(z,v) \\
    & \quad \mbox{} + 2v(1-z) \sum_{r=1}^{\ell-1} \dzp G_{r}(z,v) \cdot G_{\ell-r}(z,v).
\end{split}
\end{equation}
\end{prop}
\begin{proof}
In order to treat the distributional recurrence \eqref{MISOdistributionalY1} and to get Proposition~\ref{MISGellzvProp} it turns out to be appropriate to introduce the trivariate generating functions
$G(z,v,u) := \sum_{\ell \ge 1} G_{\ell}(z,v) u^{\ell}$.
Multiplying \eqref{MISOdistributionalY1} with $\frac{(n-1) \binom{n}{\ell}}{n} z^{n} u^{\ell}$ and taking summation over all $n \ge \ell \ge 1$ gives, after straightforward computations, the following integro-differential equation for $G(z,v,u)$:
\begin{align*}
  z \dzp G(z,v,u) - G(z,v,u) & = v \dzp G(z,v,u) \cdot \int_{0}^{z} G(z,v,u) dt + \frac{v}{1-z} \int_{0}^{z} G(t,v,u) dt \\
  & \quad \mbox{} + v \dzp G(z,v,u) \cdot \Big(z-(1-z)\lo{z}\Big),
\end{align*}
which, after simple algebraic operations, yields the following second order non-linear differential equation:
\begin{equation}\label{MISOGzvuODE}
\begin{split}
  & (1-z)^{2} G(z,v,u) \cdot \dzpZwo G(z,v,u) + (1-z) \Big(z(1-v)+v(1-z)\lo{z}\Big) \dzpZwo G(z,v,u) \\
  & \quad \mbox{} - v(1-z)^{2} G(z,v,u) \cdot \Big(\dzp G(z,v,u)\Big)^{2} - v(1-z)^{2} \lo{z} \Big(\dzp G(z,v,u)\Big)^{2} \\
  & \quad \mbox{} - 2v(1-z) G(z,v,u) \cdot \dzp G(z,v,u) - \Big(z(1-v)+2v(1-z)\lo{z}\Big) \dzp G(z,v,u) \\
  & \quad \mbox{} + (1-v) G(z,v,u) = 0.
\end{split}
\end{equation}
The stated differential equation \eqref{MISOGellzvODE} follows now from \eqref{MISOGzvuODE} by extracting coefficients, $G_{\ell}(z,v) = [u^{\ell}] G(z,v,u)$, where we omit here these straightforward computations.
\end{proof}

\subsection{Asymptotics of the moments}
Again we will apply the method of moments to show the beta-distributed limit law of a suitably normalized version of $\rvY$. Thus, we introduce, for $\ell \ge 1$ and $s \ge 0$, the functions
\begin{equation}\label{MISOGellzvDef1}
  G_{\ell,s}(z) := E_v D_v^s G_\ell(z,v) = \sum_{n \ge \ell} \frac{\binom{n}{\ell}}{n} \E\big(\rvY^{\underline{s}}\big) z^{n},
\end{equation}
with $G_{\ell}(z,v)$ defined in \eqref{MISOGellzvDef}. Therefore, the $s$-th factorial moments of $\rvY$ can be obtained as follows:
\begin{equation}\label{MISOGellzvDef2}
  \E(\fallfak{\rvY}s) = \frac{n}{\binom{n}{\ell}}[z^{n}]G_{\ell,s}(z).
\end{equation}

The following lemma collects the analytic properties of the functions $G_{\ell,s}(z)$, required to deduce the asymptotic behaviour of the moments of $\rvY$.
\begin{lemma}
\label{LemmaMomsGells}
The generating functions $G_{\ell,s}(z)= E_v D_v^s G_{\ell}(z,v)$ are, for all $\ell \ge 1$ and $s \ge 0$, $\Delta$-regular functions.
Moreover, for $\ell \ge 1$ and $s \ge 0$, $G_{\ell,s}(z)$ admits the following local expansion around the dominant singularity $z=1$:
\begin{equation}\label{MISOGellsexp}
  G_{\ell,s}(z) = \frac{\alpha_{\ell,s}}{(1-z)^{\ell+s}\Lo{z}{s}} +  \Gro\Big(\frac{1}{(1-z)^{\ell+s}\Lo{z}{s+1}}\Big),
\end{equation}
with
\begin{equation}\label{MISOalphaells}
  \alpha_{\ell,s} = \frac{s!}{\ell+s} \binom{\ell+s-1}{s}.
\end{equation}
\end{lemma}
\begin{proof}
We will show this lemma by using induction with respect to $\ell$ and $s$.
First we consider the case $s=0$. Using definition~\eqref{MISOGellzvDef1} we get
\begin{equation*}
  G_{\ell,0}(z) = \frac{1}{\ell} \frac{z^{\ell}}{(1-z)^{\ell}}, \quad \text{for $\ell \ge 1$},
\end{equation*}
thus showing that Lemma~\ref{LemmaMomsGells} holds for $s=0$.

Next we treat the case $\ell, s \ge 1$ and consider the differential equation~\eqref{MISOGellzvODE} for the functions $G_{\ell}(z,v)$ with inhomogeneous part $b_{\ell}(z,v)$ given by \eqref{MISObellzvDefG}. Applying the operator $E_{v} D_{v}^{s}$ to \eqref{MISOGellzvODE} yields the following differential equation for $G_{\ell,s}(z)$, $\ell, s \ge 1$:
\begin{equation}\label{MISOGellODE}
  (1-z)^{2} \lo{z} G_{\ell,s}''(z) - 2 (1-z) \lo{z} G_{\ell,s}'(z) = b_{\ell,s}(z),
\end{equation}
where the inhomogeneous part $b_{\ell,s}(z)$ is given by
\begin{align}
b_{\ell,s}(z) & = s (1-z) \Big(1-(1-z)-(1-z)\lo{z}\Big) G_{\ell,s-1}''(z) \notag \\
& \quad  \mbox{} - s \Big(1-(1-z)-2(1-z)\lo{z}\Big) G_{\ell,s-1}'(z) \notag \\
& \quad \mbox{} + (1-z)^{2} \lo{z} \sum_{r=1}^{\ell-1} \sum_{s_{1}+s_{2}=s, s_{1}, s_{2} \ge 0}
\binom{s}{s_{1},s_{2}} G_{r,s_{1}}'(z) G_{\ell-r,s_{2}}'(z) \notag \\
& \quad \mbox{} + s G_{\ell,s-1}(z) \notag \\
& \quad \mbox{} - (1-z)^{2} \sum_{r=1}^{\ell-1} \sum_{s_{1}+s_{2}=s, s_{1}, s_{2} \ge 0} \binom{s}{s_{1}, s_{2}} G_{r,s_{1}}''(z) G_{\ell-r,s_{2}}(z) \label{MISObellDefG} \\
& \quad \mbox{} + (1-z)^{2} \sum_{r=1}^{\ell-1} \sum_{q=1}^{\ell-r-1} \sum_{s_{1}+s_{2}+s_{3}=s, s_{1}, s_{2}, s_{3} \ge 0} \binom{s}{s_{1}, s_{2}, s_{3}} G_{r,s_{1}}'(z) G_{q,s_{2}}'(z) G_{\ell-r-q,s_{3}}(z) \notag \\
& \quad \mbox{} + s (1-z)^{2} \sum_{r=1}^{\ell-1} \sum_{q=1}^{\ell-r-1} \sum_{s_{1}+s_{2}+s_{3}=s-1, s_{1}, s_{2}, s_{3} \ge 0} \binom{s-1}{s_{1}, s_{2}, s_{3}} G_{r,s_{1}}'(z) G_{q,s_{2}}'(z) G_{\ell-r-q,s_{3}}(z) \notag \\
& \quad \mbox{} + s (1-z)^{2} \lo{z} \sum_{r=1}^{\ell-1} \sum_{s_{1}+s_{2}=s-1, s_{1}, s_{2} \ge 0} \binom{s-1}{s_{1}, s_{2}} G_{r,s_{1}}'(z) G_{\ell-r,s_{2}}'(z) \notag \\
& \quad \mbox{} + 2(1-z) \sum_{r=1}^{\ell-1} \sum_{s_{1}+s_{2}=s, s_{1}, s_{2} \ge 0} \binom{s}{s_{1}, s_{2}} G_{r,s_{1}}'(z) G_{\ell-r,s_{2}}(z) \notag \\
& \quad \mbox{} + 2s(1-z) \sum_{r=1}^{\ell-1} \sum_{s_{1} + s_{2} = s-1, s_{1}, s_{2} \ge 0} \binom{s-1}{s_{1}, s_{2}} G_{r,s_{1}}'(z) G_{\ell-r,s_{2}}(z). \notag
\end{align}

The differential equation~\eqref{MISOGellODE} can be solved easily; below we state the particular solution, which satisfies the initial conditions $G_{\ell,s}(0)=0$ and $G_{\ell,s}'(0)=0$ and thus is indeed the required solution (for $\ell, s \ge 1$):
\begin{equation}\label{MISOGellssol}
  G_{\ell,s}(z) = \int_{0}^{z} \frac{1}{(1-t)^{2}} \left(\int_{0}^{t} \frac{b_{\ell,s}(u)}{\lo{u}} du\right) dt,
\end{equation}
with $b_{\ell,s}(z)$ defined in \eqref{MISObellDefG}.

Again we observe that the representation \eqref{MISOGellssol} together with the closure properties for singular differentiation and integration inductively shows that all $G_{\ell,s}(z)$, $\ell,s \ge 1$, are $\Delta$-regular functions. Note that it is known a priori from the definition of $N_{\ell,s}(z)$ and simple majorization arguments that $N_{\ell,s}(z)$ is analytic for $|z| < 1$, so we do not have to take care about the analyticity of $G_{\ell,s}(z)$ around $z=0$ (which, of course, can also be obtained easily from \eqref{MISOGellssol} by showing that $b_{\ell,s}(0)=0$).

We proceed by showing in an inductive way the local expansions \eqref{MISOGellsexp} of $G_{\ell,s}(z)$ in a complex neighbourhood of $z=1$. 
To do this we consider $\ell, s \ge 1$ and assume that \eqref{MISOGellsexp} holds for all $(j,r)$ with $1 \le j \le \ell$, $0 \le r \le s$ and $(j,r) \neq (\ell,s)$.
When examining each summand of $b_{\ell,s}(z)$ as given in \eqref{MISObellDefG} and using the induction hypothesis as well as Lemma~\ref{MISOsingana}, it turns out that only the first three summands of \eqref{MISObellDefG} give major contributions, which are stated below:
\begin{align*}
  & \bullet \; s (1-z) \Big(1-(1-z)-(1-z)\lo{z}\Big) G_{\ell,s-1}''(z) \\
  & \quad = \frac{s(\ell+s-1) (\ell+s) \alpha_{\ell,s-1}}{(1-z)^{\ell+s} \Lo{z}{s-1}} \cdot \Big(1+\Big(\frac{1}{\lo{z}}\Big)\Big), \\
  & \bullet \; -s \Big(1-(1-z)-2(1-z)\lo{z}\Big) G_{\ell,s-1}'(z) \\
  & \quad = \frac{-s(\ell+s-1) \alpha_{\ell,s-1}}{(1-z)^{\ell+s} \Lo{z}{s-1}} \cdot \Big(1+\Big(\frac{1}{\lo{z}}\Big)\Big), \\
  & \bullet \; (1-z)^{2} \lo{z} \sum_{r=1}^{\ell-1} \sum_{s_{1}+s_{2}=s, s_{1}, s_{2} \ge 0} \binom{s}{s_{1},s_{2}} G_{r,s_{1}}'(z) G_{\ell-r,s_{2}}'(z) \\
  & \quad = \frac{\sum_{r=1}^{\ell-1} \sum_{s_{1}+s_{2}=s, s_{1}, s_{2} \ge 0} \binom{s}{s_{1}, s_{2}} (r+s_{1}) \alpha_{r,s_{1}} (\ell-r+s_{2}) \alpha_{\ell-r,s_{2}}}{(1-z)^{\ell+s} \Lo{z}{s-1}} \cdot \Big(1+\Big(\frac{1}{\lo{z}}\Big)\Big),
\end{align*}
whereas the contribution of the remaining terms of $b_{\ell,s}(z)$ is of order $\mathcal{O}\left(\frac{1}{(1-z)^{\ell+s} \Lo{z}{s}}\right)$.
Adding these contributions we obtain that $b_{\ell,s}(z)$ has the following local expansion around $z=1$:
\begin{equation}\label{MISObellsexpG}
\begin{split}
  b_{\ell,s}(z) & = \frac{s (\ell+s-1)^{2} \alpha_{\ell,s-1} + \sum_{r=1}^{\ell-1} \sum_{s_{1}+s_{2}=s, s_{1}, s_{2} \ge 0} \binom{s}{s_{1}, s_{2}} (r+s_{1}) \alpha_{r,s_{1}} (\ell-r+s_{2}) \alpha_{\ell-r,s_{2}}}{(1-z)^{\ell+s} \Lo{z}{s-1}}\\
  & \quad \times \Big(1+\Big(\frac{1}{\lo{z}}\Big)\Big).
\end{split}
\end{equation}
Using the representation \eqref{MISOGellssol} and \eqref{MISObellsexpG} yields after applications of singular integration the following local expansion of $G_{\ell,s}(z)$:
\begin{equation}\label{MISOGellszstruct}
  G_{\ell,s}(z) = \frac{\alpha_{\ell,s}}{(1-z)^{\ell+s} \Lo{z}{s}} \cdot \Big(1+\Big(\frac{1}{\lo{z}}\Big)\Big),
\end{equation}
where the numbers $\alpha_{\ell,s}$ satisfy the following recurrence:
\begin{equation}\label{MISOalphaellsrec}
  \alpha_{\ell,s} = \frac{s (\ell+s-1)^{2} \alpha_{\ell,s-1} + \sum_{r=1}^{\ell-1} \sum_{s_{1}+s_{2}=s, s_{1}, s_{2} \ge 0} \binom{s}{s_{1},s_{2}} (r+s_{1}) \alpha_{r,s_{1}} (\ell-r+s_{2}) \alpha_{\ell-r, s_{2}}}{(\ell+s) (\ell+s-1)}.
\end{equation}
Plugging the induction hypothesis \eqref{MISOalphaells} for all $\alpha_{j,r}$ with $(j,r) < (\ell,s)$ into the right hand side of \eqref{MISOalphaellsrec} yields, after an application of the Vandermonde convolution formula, that $\alpha_{\ell,s} = \frac{s!}{\ell+s} \binom{\ell+s-1}{s}$ also holds.
Thus the expansion \eqref{MISOGellsexp} is also valid for $(\ell,s)$; this completes the proof of Lemma~\ref{LemmaMomsGells}.
\end{proof}

Applying singularity analysis to the expansion of $G_{\ell,s}(z)$ given in Lemma~\ref{LemmaMomsGells} together with the definition~\eqref{MISOGellzvDef1} immediately shows the following asymptotic growth behaviour of the $s$-th factorial moments of $\rvY$ (with $\ell \ge 1$, $s \ge 0$):
\begin{equation}\label{MISOYnellfakmom}
   \E(\fallfak{\rvY}{s}) = \frac{n}{\binom{n}{\ell}} [z^{n}] G_{\ell,s}(z) = \frac{\ell}{\ell+s} \Big(\frac{n}{\log n}\Big)^{s} \cdot \Big(1+\mathcal{O}\Big(\frac{1}{\log n}\Big)\Big).
\end{equation}
Using \eqref{MISOfallfakordmom} we obtain the first part of Theorem~\ref{MISOthe3}. 
This implies that, after suitable normalization, the $s$-th integer moments of $\rvY$ converge to the moments of a beta-distributed random variable with parameters $1$ and $\ell$:
\begin{equation*}
\E\Big(\big(\frac{\log n}{n} \rvY\big)^s\Big) =  \frac{\ell}{\ell+s} + \mathcal{O}\Big(\frac{1}{\log n}\Big),
\end{equation*}
which also proves the limit law stated in Theorem~\ref{MISOthe3}.

\end{document}